\documentclass[12pt]{amsart}

\usepackage[english]{babel}
\usepackage[cp1250]{inputenc}
\usepackage{graphicx}
\usepackage{url}
\usepackage{multicol}
\usepackage{amsthm}
\usepackage{amsmath}
\usepackage{amssymb}
\usepackage{geometry}
\usepackage{pinlabel}
\usepackage{hyperref}
\usepackage{xcolor}
\usepackage[all]{xy}
\usepackage{mathtools}

\newtheorem{Theorem}{Theorem}
\newtheorem{proposition}[Theorem]{Proposition}
\newtheorem{corollary}[Theorem]{Corollary}
\newtheorem{remark}[Theorem]{Remark}
\newtheorem{example}[Theorem]{Example}

\theoremstyle{definition}
\newtheorem{definition}{Definition}

\newtheorem{lemma}{Lemma}

\newcommand{\RR}{\mathbb {R}}

\newcommand*\xbar[1]{%
   \hbox{%
     \vbox{%
       \hrule height 0.5pt 
       \kern0.5ex
       \hbox{%
         \kern-0.1em
         \ensuremath{#1}%
         \kern-0.1em
       }%
     }%
   }%
}

\def\aut{{\mathrm{Aut}}}

\def\aff#1{\mathrm{Aff}#1}

\def\aff#1{\mathrm{Aff}#1}

\def\dis{\mathrm{Dis}}
\def\inn{{\mathrm{Inn}}}

\def\Alex{\mathrm{Alex}}

\def\setof#1#2{\{#1\, : \,#2\}}

\def\Z{\mathbb Z}

\def \up {\, \overline{\ast }\, }

\begin{document}
\title{Knot quandle decomposition along a torus}

\author{Marco Bonatto, Alessia Cattabriga, Eva Horvat}

\address[Bonatto]{Dipartimento di Matematica e Informatica \\ UNIFE}
\email[Bonatto]{marco.bonatto.87@gmail.com}

\address[Cattabriga]{Department of Mathematics \\University of Bologna\\Piazza di Porta San Donato 5\\ 40126 Bologna \\Italy}
\email{alessia.cattabriga@unibo.it}

\address[Horvat]{University of Ljubljana\\
Faculty of Education\\
Kardeljeva plo\v s\v cad 16\\
1000 Ljubljana, Slovenia}
\email[Horvat]{eva.horvat@pef.uni-lj.si}

\keywords{augmented fundamental quandle, links in the solid torus, satellite knots, links in lens spaces, Alexander quandle, quandle colorings}
\subjclass[2020]{57K10, 57K12, 57K14}
\date{\today}

\begin{abstract}
We study the structure of the augmented fundamental quandle of a knot whose complement contains an incompressible torus. We obtain the relationship between the fundamental quandle of a satellite knot and the fundamental quandles/groups of its companion and pattern knots. General presentations of the fundamental quandles of a link in a solid torus, a link in a lens space and a satellite knot are described. In the last part of the paper, an algebraic approach to the study of affine quandles is presented and some known results about the Alexander module and quandle colorings are obtained. 
\end{abstract}
\maketitle

\begin{section}{Introduction}\label{sec1}
Forty years ago, quandles were introduced independently by David Joyce  \cite{JO} and Sergei Matveev \cite{Ma} as an algebraic structure that generalizes the knot group. Since then, a considerable amount of research has aimed at using the algebra of quandles in knot theory, while on the other hand, quandles and related structures have been intensively studied as abstract algebraic objects. In this setting, it would be interesting to explore how the machinery  developed in the algebraic setting could be used to obtain and recover topological results.\\

 In \cite{CH}, we explored the  fundamental quandle as a functor from the category of oriented tangles to a suitably defined quandle category. This implied results about construction of the fundamental quandle of a link from the fundamental quandles of its constituent tangles. \\

In this paper, we are looking at a different kind of  knot quandle decomposition. For any knot $K$ in a closed orientable connected 3-manifold $M$, the action of the fundamental group on the knot quandle $Q(K)$ implies that $Q(K)$ has a natural structure of an augmented quandle \cite{FR}. If $M$ is not a simply connected manifold, the augmentation map $\epsilon \colon Q(K)\to \pi _{1}(K)$ is not surjective. In such case, the augmented fundamental quandle of $K$ is given by a general presentation, as defined in \cite{FR}. Our study is concerned with the case when the complement of a knot $K$ contains an incompressible torus $T$. Then the fact that $K$ is notrivially embedded in a solid torus with boundary $T$ is reflected in the structure of its fundamental quandle. Specifically, the fundamental quandle of $K$ may be expressed as an augmented quandle, whose underlying set is the quandle of a knot in the solid torus. In the classical case of knots in $S^{3}$, this happens when $K$ is a satellite knot. In Theorem \ref{th2}, the fundamental quandle of a satellite knot is expressed in terms of the fundamental quandle of its pattern and the fundamental groups of the pattern and companion knots.  Moreover, when $M$ is a lens space, then the complement of every non-local knot $K$ in $M$ contains an incompressible torus, and a presentation for  the augmented fundamental quandle of $K$ may be expressed in a simple way. Since lens spaces may be thought of as ``building blocks'' for more complicated 3-manifolds, a similar technique could be applied to obtain a presentation for the fundamental quandle of a knot inside any 3-manifold $M$.  \\

Despite finding a presentation for the knot quandle is not difficult, the general algebraic structure of knot quandles remains quite elusive. As in group theory, abelian groups are well understood due to their $\mathbb Z$-module structure,  affine quandles represent a better known and probably most studied class of quandles, since they are essentially modules over the ring of Laurent polynomials. Among them, the most important from a topological point of view is the classical Alexander module of a knot.  In the paper, we associate to each connected quandle an affine quandle that is the quandle analogue of group abelianization and that in the case of knot quandles equals the Alexander module. Using this approach and the presentation of the quandle of a satellite knot, we recover a classical result of \cite{LM} about the structure of the module and some results on colorings by affine quandles (see \cite{BAE}).\\

The paper is structured as follows. In Section \ref{sec2}, we recall some background on quandles themselves and on the fundamental quandle of links. Subsection \ref{sub21} contains the basic definitions of a quandle, an augmented quandle and related terms, while Subsection \ref{sub22} recalls quandle presentations. In Subsection \ref{sub23} we explain the topological meaning of the fundamental quandle of a link and describe its presentation. In Section \ref{sec3}, we obtain a general presentation of the augmented fundamental quandle of a link in the solid torus. This result is applied in Section \ref{sec4} to present the quandle of a link in a lens space. In Section \ref{sec5}, we focus on satellite knots. First we recall the relationship between the fundamental group of a satellite knot and the groups of its companion and pattern knots. Then we express the fundamental quandle of a satellite knot in terms of the fundamental quandles/groups of its companion and pattern knots. We obtain a general presentation of the fundamental quandle of a satellite knot, which is illustrated in a couple of sample computations. In Section \ref{sec6}, we turn our attention to affine quandles. A quick recollection of basic definitions and properties is done in Subsection \ref{sub61}. In Subsection \ref{sub62}, we describe the construction of the Alexander module of a connected quandle, and its relationship with the Alexander module of a knot.  Subsection \ref{sub63} contains results about colorability of knots by affine quandles.

\end{section}

\section{Preliminaries}\label{sec2}

\subsection{Quandles as algebraic structures}\label{sub21}
A \textit{quandle} is a nonempty set $Q$ with a binary operation $\ast \colon Q\times Q\to Q$ that satisfies the following axioms:
\begin{enumerate}
\item $x\ast x=x$ for every $x\in Q$.
\item The mapping $R_{y}\colon Q\to Q$, called \textit{right translation} given by $R_{y}(x)=x\ast y$, is a bijection for every $y\in Q$. 
\item For any $x,y,z\in Q$ the formula $(x\ast y)\ast z=(x\ast z)\ast (y\ast z)$ holds.  
\end{enumerate} A set $Q$ with a binary operation that satisfies only the axioms (2) and (3) is called a \textit{rack}.

Using  axiom (2), it is possible to see that a quandle $(Q,*)$ admits another binary operation $\up $, given by $x\up y:=R_{y}^{-1}(x)$ for every $x,y\in Q$. The equalities $(x*y)\up y=(x\up y)*y=x$ hold for every $x,y\in Q$. Every group $G$ has a natural quandle structure, given by conjugation, that is $g*h=hg h^{-1}$, that will be denoted as $\mathrm{Conj}(G)$.

A map between quandles $f\colon (Q_1,*_1)\to (Q_2, *_2)$ is called a \textit{quandle homomorphism} if it satisfies $f(x\ast _{1}y)=f(x)\ast _{2}f(y)$ for every $x,y\in Q_{1}$. Note that in any quandle $Q$, the right translation $R_{y}(x)=x*y$ is a quandle automorphism. 
 An equivalence relation $\sim $ on a quandle $(Q,*)$ is called a \textit{congruence} if the implication $$x\sim y\textrm{ and }z\sim w \Rightarrow x*z\sim y*w\textrm{ and }x\up z\sim y\up w$$ holds for every $x,y,z,w\in Q$. Every quandle homomorphism $f:Q\to Q'$ defines a congruence on $Q$ by $x\sim y \iff f(x)=f(y)$.

The subgroup of the automorphism  group\footnote{For function composition we take the convention  $fg(x)=f(g(x))$.}, $\aut(Q)$, generated by all right translations is called the \textit{inner automorphism group} $\inn(Q)$, and the subgroup 
$$\dis(Q)=\langle R_x R_y^{-1}\, : \, x,y\in Q\rangle$$ of $\inn(Q)$ is called the {\it displacement group of $Q$}. It is a well-known fact that the groups $\inn(Q)$ and $\dis(Q)$ have the same orbits: if they act transitively on $Q$, the quandle is said to be {\it connected}. Many quandle theoretical properties are determined by the group theoretical properties of the displacement group. If $h:Q\longrightarrow Q'$ is a surjective quandle homomorphism, then the function 
$$R_x \mapsto R_{h(x)} $$
defined on generators of $\inn(Q)$ can be extended  to a surjective group homomorphism $\inn(Q)\longrightarrow \inn(Q')$, and such homomorphism restricts and corestricts to the displacement groups.\\ 

In this paper we are interested in quandles associated to topological structures: in this setting, the notion of augmented quandle arises naturally. We recall it from \cite{JO}. 

\begin{definition}\label{def2} An \textit{augmented quandle} is given by a pair $(Q,G)$, where $G$ is a group, equipped with a right action on a set $Q$ (which we will write exponentially), and there is a map $\epsilon \colon Q\to G$ (the \textit{augmentation map}) that satisfies \begin{enumerate}
\item $q^{\epsilon (q)}=q$, $\forall q\in Q$,
\item $\epsilon (q^x)=x^{-1}\epsilon (q)x$, $\forall q\in Q$, $\forall x\in G$. 
\end{enumerate}
The augmented quandle $(Q,G)$ is the set $Q$ with a binary operation $*$ that is given by $x*y=x^{\epsilon (y)}$ for any $x,y\in Q$. 
\end{definition}

It is easy to see that an augmented quandle is, indeed, a quandle. Moreover, every quandle $Q$ comes equipped with a natural augmentation map $\epsilon \colon Q\to \inn(Q)$, defined by $\epsilon (q)=R_{q}$.

\begin{definition}\label{def6} A \textit{homomorphism} of augmented quandles $(Q,G)$ and $(P,H)$ consists of a group homomorphism $g\colon G\to H$ and a function $f\colon Q\to P$, such that the diagram 
\begin{displaymath}
\xymatrix{
Q\times G \ar@{->}[d]^{f\times g} \ar@{->}[r] & Q \ar@{->}[r]^{\epsilon} \ar@{->}[d]^{f} & G \ar@{->}[d]^{g}\\
P\times H \ar@{->}[r] & P \ar@{->}[r]^{\epsilon} & H}
\end{displaymath}
commutes.
\end{definition}

\subsection{Quandle presentations}\label{sub22}
We recall the basics about quandle presentations from \cite{FR}.

\begin{definition}\label{def3} Denote by $F(S)$ the free group, generated by a set $S$. Let $\mathfrak{ip}$ be the smallest equivalence relation on the product $S\times F(S)$, such that $(a,w)\mathfrak{ip}(a,aw)$ for every $(a,w)\in S\times F(S)$. The \textit{free quandle} on $S$ is the augmented quandle $FQ(S)=((S\times F(S))/_{\mathfrak{ip}},F(S))$, where the action of $F(S)$ on $(S\times F(S))/_{\mathfrak{ip}}$ is on the second entry, and the augmentation map $\epsilon \colon (S\times F(S))/_{\mathfrak{ip}}\to F(S)$ is given by $\epsilon ([a,w])=w^{-1}aw$. It follows that the quandle operation is given by $[a,w]*[b,z]=[a,wz^{-1}bz]$ for any $a,b\in S$ and $w,z\in F(S)$.    
\end{definition}
The free quandle generated by the set $S$ has the universal property that for any quandle $Q$, a given function $S\to Q$ extends uniquely to a quandle homomorphism $FQ(S)\to Q$. \\

A \textit{primary presentation} of a quandle consists of two sets $S$ (the generating set) and $R$ (the set of relations). Elements of $R$ are ordered pairs $(x,y)$ with $x,y\in FQ(S)$, which we will write as equations: $x=y$. The presentation defines a quandle $[S\, \colon R]$ as follows. Let $\sim $ be the smallest congruence on the free quandle $FQ(S)$ containing $R$ (such that $x\sim y$ whenever $(x=y)\in R$). Then $$[S\, \colon R]:=\frac{FQ(S)}{\sim }\;.$$

The notation $[S:R]$ will be reserved for  quandle presentations, while for group presentations we will use the more standard notation $\langle S:R\rangle$.

In order to deal with fundamental quandles of knots in non simply connected manifolds, we will need the notion of a general quandle presentation.  

\begin{definition} Given two sets $S$ and $T$, the \textit{extended free quandle} $FQ(S,T)$ is the augmented quandle $FQ(S,T)=((S\times F(S\cup T))/_{\mathfrak{ip}},F(S\cup T))$, where $F(S\cup T)$ acts by the canonical right action on the second entry of $(S\times F(S\cup T))/_{\mathfrak{ip}}$, and whose augmentation map is given by $\epsilon [a,w]=w^{-1}aw$. The quandle operation is thus given by $[a,w]*[b,z]=[a,wz^{-1}bz]$ for any $a,b\in S$ and $w,z\in F(S\cup T)$. 
\end{definition}
 
A \textit{general presentation} of a quandle consists of four sets: the primary generating set $S_{P}$, the operator generators $S_{O}$, the primary relations $R_{P}$ and the operator relations $R_{O}$. Elements of $R_{P}$ are statements of the form $x=y$, where $x,y\in FQ(S_{P},S_{O})$, and elements of $R_{O}$ are words $w\in F(S_{P}\cup S_{O})$.  Given a general presentation $[S_{P},S_{O}\colon R_{P},R_{O}]$, let $\sim $ be the smallest congruence on $FQ(S_{P},S_{O})$ that contains
\begin{enumerate}
\item $x\sim y$ for every $(x=y)\in R_{P}$,
\item $[z,w]\sim [z,1]$ for every $z\in S_{P}$ and $w\in R_{O}$.
\end{enumerate}

Then the quandle, generated by the presentation, is defined as 
\begin{equation}\label{general presentation}
 [S_{P},S_{O}\colon R_{P},R_{O}]:=\frac{FQ(S_{P},S_{O})}{\sim }\;.
\end{equation}

\begin{remark}\label{remark on generators}
Let $Q$ be the quandle, presented as in \eqref{general presentation}.
Note that the elements of $Q$ are of the form
\begin{equation*}\label{elements for general presentations}
[(x,g)]=[(x,1)]\cdot g
\end{equation*}
for $x\in S_P$ and $g\in F(S_P\cup S_O)$.
\end{remark}

\subsection{The fundamental quandle of a link}\label{sub23}

The fundamental quandle of a codimension 2 link was defined in \cite{FR}. For the reader's convenience, we include a brief review of its construction. 

Let $M$ be a closed connected oriented manifold, and let $L\subset M$ be a nonempty submanifold of codimension 2. We assume that the embedding is proper if either manifold contains boundary, and that $L$ is transversely oriented in $M$. Denote by $N_{L}$ a regular neighbourhood of $L$ inside $M$, and by $E_{L}=\mathrm{closure}(M-N_{L})$ its exterior. Choose a basepoint $z\in E_{L}$ and define 
$$\mathcal P_{L}=\{\textrm{homotopy classes of paths in $E_{L}$ from $\partial N_{L}$ to $z$}\}\;.$$ 
The homotopy has to fix $z$ and may move the starting point along $\partial N_L$. The fundamental group $\pi _{1}(E_{L})=\pi _{1}(E_{L}, z)$ has a natural right action on the set $\mathcal P_{L}$. If $a\in \mathcal P_{L}$ is represented by a path $\alpha $ and $g\in \pi _{1}(E_{L})$ is represented by a loop $\gamma $, then $a^{g}$ is defined to be the homotopy class of the composite path $\alpha \gamma $. 

For any point $x\in \partial N_{L}$, denote by $m_{x}$ the loop based at $x$, which runs once around the meridian of $L$ in the positive direction. Now define a map $\epsilon \colon \mathcal P_{L}\to \pi _{1}(E_{L})$ by $$\epsilon ([\alpha ])=[\overline{\alpha }m_{\alpha (0)}\alpha ]\;,$$ where $[\zeta ]$ denotes the homotopy class, and $\overline{\zeta }$ denotes the reversal of a path $\zeta $. It is easy to see that the pair $(\mathcal P _{L},\pi _{1}(E_{L}))$ defines an augmented quandle with augmentation map $\epsilon $ (see Definition \ref{def2}). 

\begin{definition} \label{def5} The augmented quandle $(\mathcal P_{L},\pi _{1}(E_{L}))$ defined above is called \textit{the fundamental quandle} of the link $L$, and is denoted by $Q(L)$. 
\end{definition}

For the sake of simplicity, we denote an element $[\alpha]$  of $Q(L)$ just as $\alpha$. The correspondence $(M,L)\to Q(L)$ is functorial, that is   a map  $f:(M,L)\to (M',L')$  induces an extended quandle homomorphism at the augmented quandle level, that we still denote with $f$.

The fundamental quandle of a link in $S^{3}$ admits a simple presentation in terms of a link diagram. Let $D_{L}$ be an oriented link diagram of a link $L$ in $S^{3}$. Figure \ref{fig:crossingQ} depicts the \textit{crossing relation} at a crossing of the diagram $D_{L}$.

\begin{figure}[h]
\labellist
\normalsize \hair 2pt
\pinlabel $x$ at 10 150
\pinlabel $y$ at 180 150
\pinlabel $x*y$ at 200 40
\pinlabel $y$ at 610 150
\pinlabel $x\up y$ at 580 40
\pinlabel $x$ at 780 150
\endlabellist
\begin{center}
\includegraphics[scale=0.30]{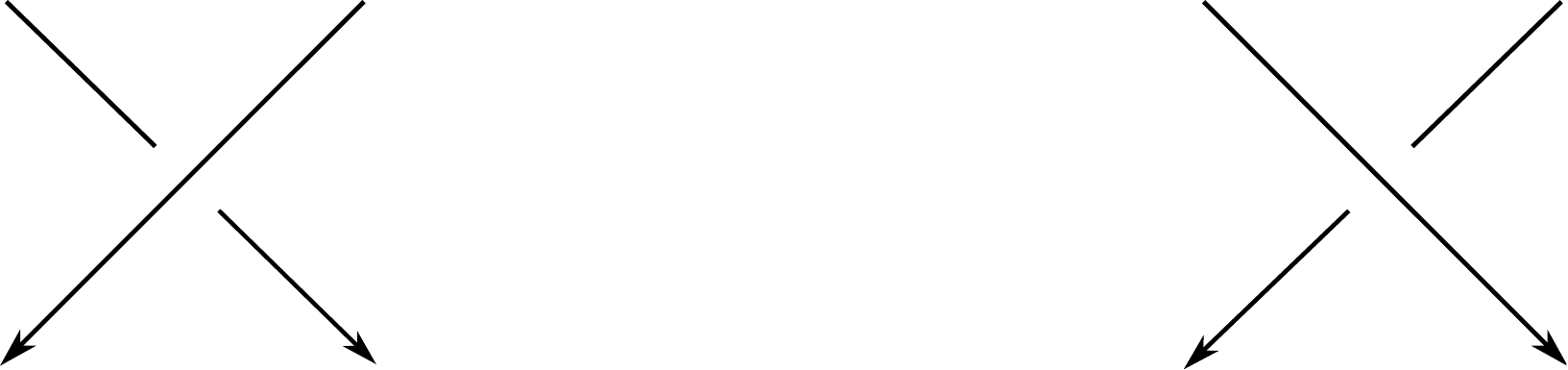}
\caption{The crossing relation at a positive crossing (left) and at a negative crossing (right)}
\label{fig:crossingQ}
\end{center}
\end{figure}

\begin{Theorem}\cite[Theorem 4.7]{FR}\label{th0} Given a diagram $D_{L}$ of a link $L$ in $S^{3}$, let $A(D_L)$ be the set of arcs and let $C(D_L)$ be the set of crossing relations at the crossings of $D_L$. Then the fundamental quandle $Q(L)$ of the link $L$ is given by the quandle presentation $[A(D_L)\colon C(D_L)]$. 
\end{Theorem}

Now consider a more general situation. Let $L$ be a link in a closed connected orientable 3-manifold $M$. By the Lickorish--Wallace theorem, $M$ may be obtained by rational surgery on a framed link in $S^{3}$ (see \cite[\S 16]{PS}). The link $L$ may thus be given by a diagram $D_{L}$, which contains also the surgery curves for $M$, each  decorated with the framing parameter belonging to $\mathbb Q\cup\{1/0\}$. In this diagram, the surgery curves will be colored by red, and the curves belonging to $L$ will be colored by black. Denote by $Ar(D_L)$ (respectively $Ab(D_L)$) the set of red arcs (respectively black arcs) of the diagram $D_{L}$. There are four types of crossings that may occur in the diagram $D_{L}$, and each of them determines a corresponding crossing relation as depicted in Figure \ref{fig:crossingS}. Denote by $C_{i}(D_L)$ the set of crossing relations at all crossings of type $C_{i}$ in the diagram $D_{L}$ for $i=1,2,3,4$. 

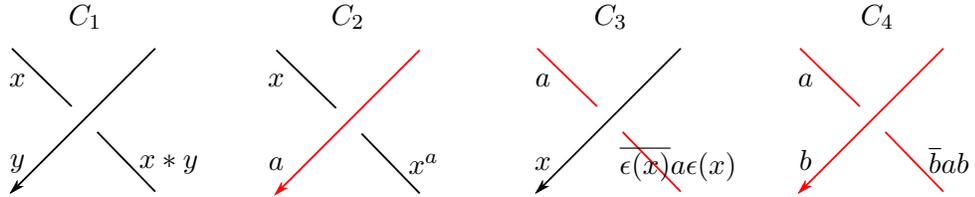
\begin{figure}[h]
\labellist
\normalsize \hair 2pt
\pinlabel $x$ at 10 170
\pinlabel $y$ at 10 65
\pinlabel $x*y$ at 200 65
\pinlabel $C_1$ at 95 250
\pinlabel $x$ at 335 170
\pinlabel $a$ at 335 65
\pinlabel $x^{a}$ at 520 65
\pinlabel $C_2$ at 425 250
\pinlabel $a$ at 670 170
\pinlabel $x$ at 670 65
\pinlabel $\overline{\epsilon (x)}a\epsilon (x)$ at 840 65
\pinlabel $C_3$ at 755 250
\pinlabel $a$ at 1000 170
\pinlabel $b$ at 1000 70
\pinlabel $\overline{b}ab$ at 1180 70
\pinlabel $C_4$ at 1090 250
\endlabellist
\begin{center}
\includegraphics[scale=0.30]{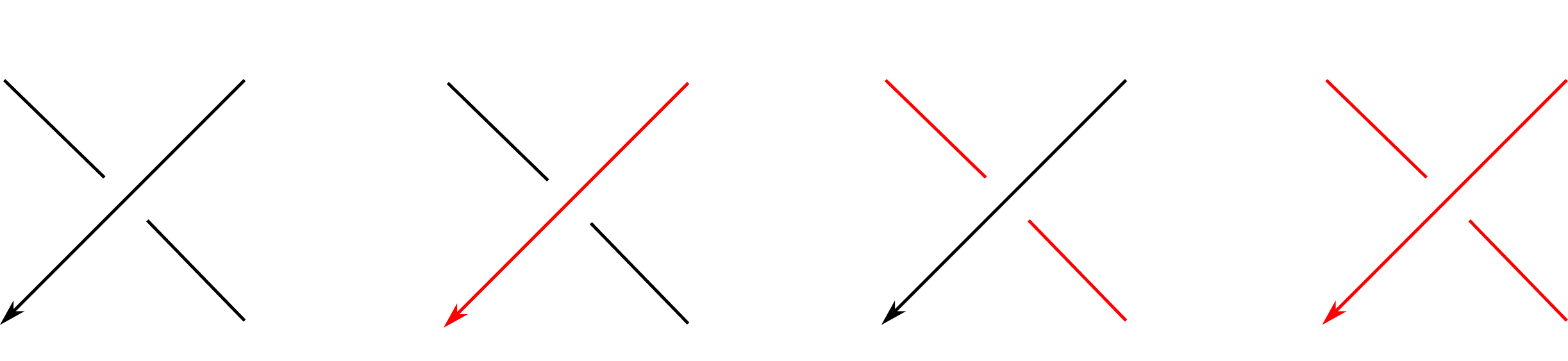}
\caption{The four types of crossing relations in a diagram of a link in a 3-manifold}
\label{fig:crossingS}
\end{center}
\end{figure}

The exterior $E_{L}=\mathrm{closure}(M-N_{L})$ of the link $L$ in the 3-manifold $M$ is obtained from $S^{3}$ by removing a regular neighborhood of each surgery curve, gluing back the solid torus according to the surgery prescription  (which means we attach a 2-handle according to the framing parameter  and  close up with the remaining 3-handle), and removing a regular neighborhood of each component of the link $L$. For each surgery curve, the attachment of the 2-handle adds a new relation to the fundamental group $\pi _{1}(E_{L})$. Denote by $R$ the set of relations in $\pi _{1}(E_{L})$ that correspond to attaching maps of the 2-handles. In \cite[Theorem 4.10]{FR}, a presentation of the fundamental rack of $L$ is described in terms of arcs and crossings of $D_L$ for the case of integer surgery (i.e. the framing parameters belong to $\mathbb Z\cup\{1/0\}$). That result easily generalizes to the following.     

\begin{Theorem} \cite[Theorem 4.10]{FR} \label{th1} The fundamental quandle $Q(L)$ of the link $L$ in a closed, orientable, connected 3-manifold $M$ is given by the general quandle presentation 
$$\left [Ab(D_L), Ar(D_L)\colon C_{1}(D_L)\cup C_{2}(D_L), C_{3}(D_L)\cup C_{4}(D_L)\cup R\right ]\;.$$
\end{Theorem}

\section{Fundamental quandle of a link in the solid torus}\label{sec3}
In this section, we describe the fundamental quandle of a link inside the simplest 3-manifold that is not the 3-sphere. This will be our building block for the fundamental quandles of satellite links and of links in lens spaces. \\

Let $L$ be a link in the solid torus $S^{1}\times D^{2}$. Imagine the solid torus standardly embedded in $\RR ^{3}$, so that its core coincides with the unit circle in $\RR ^{2}\times \{0\}$. Then a regular projection of $L$ to the plane $\RR ^{2}\times \{0\}$ with a hole in its center represents a diagram of $L$, see Figure \ref{fig1}. Using such diagrams, a knot atlas of knots in the solid torus was constructed in \cite{GA}. \\
\begin{figure}[h!]
\labellist
\normalsize \hair 2pt
\pinlabel $L$ at 510 310
\endlabellist
\begin{center}
\includegraphics[scale=0.2]{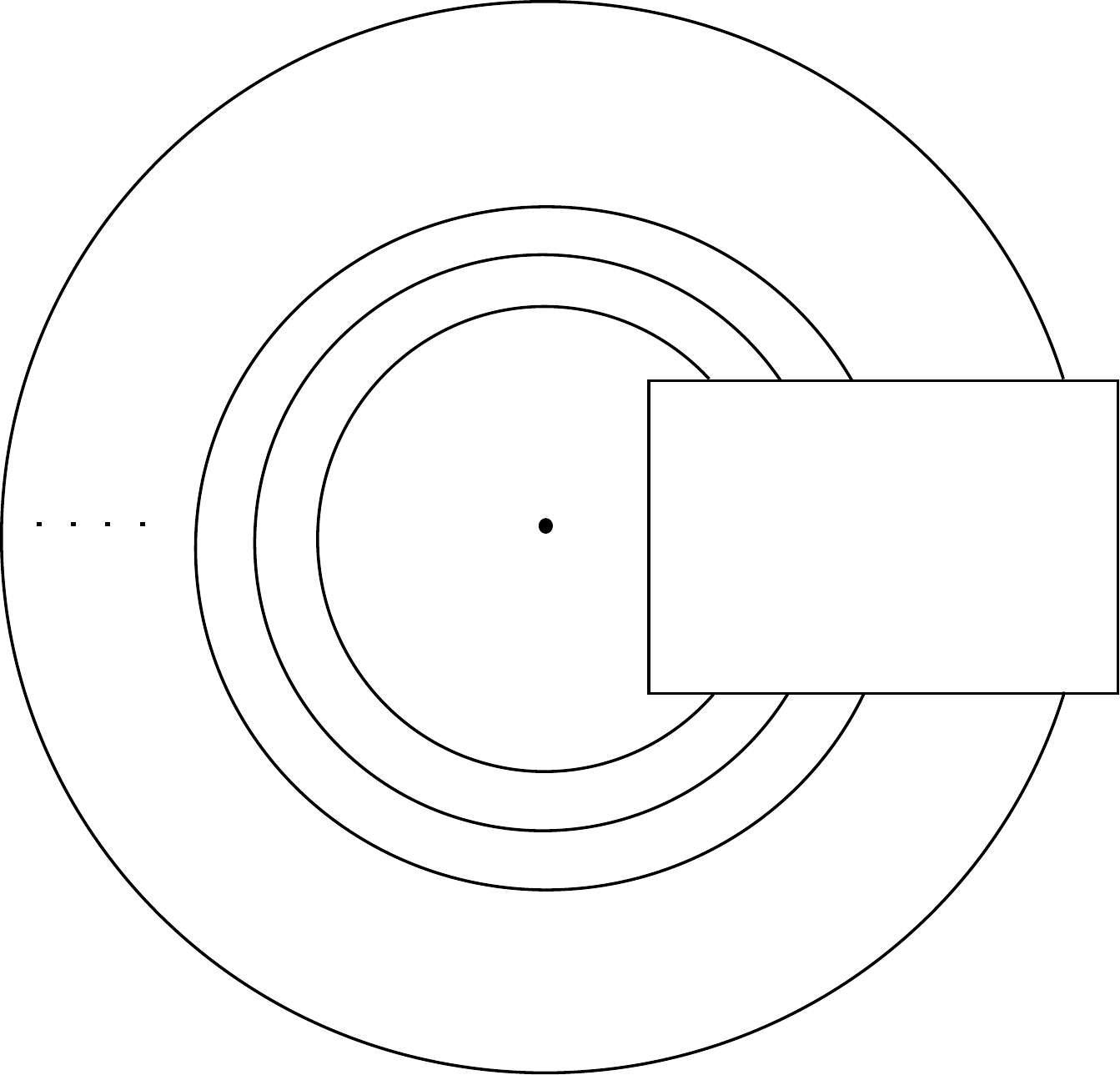}
\caption{A diagram of a link in the solid torus}
\label{fig1}
\end{center}
\end{figure}

Compactifying $\RR ^{3}$, the $z$-axis becomes a circle $C$. The solid torus $S^{1}\times D^{2}$ is obtained by cutting out a regular neighborhood of $C$ from $S^{3}$. We obtain a diagram of $L$ inside $S^{3}$, in which $C$ represents a surgery curve, see Figure \ref{fig2}.  A presentation of the fundamental quandle of $L$ may be read from this diagram. Denote by $a_{1},\ldots ,a_{d}$ the arcs of the surgery curve $C$ and by $x_{1},\ldots ,x_{n}$ the arcs of the link $L$, as depicted in Figure \ref{fig2}. Denote by $\delta _{i}\in \{-1,1\}$ the sign of the crossing with overcrossing arc $a_{1}$ and undercrossing arcs $x_{i}, x_{d+i}$. 
\begin{figure}[h!]
\labellist
\normalsize \hair 2pt
\pinlabel $L$ at 520 230
\pinlabel $x_{1}$ at 245 250
\pinlabel $x_{2}$ at 205 270
\pinlabel $x_{3}$ at 150 290
\pinlabel $x_{d}$ at 35 280
\pinlabel $x_{d+1}$ at 350 150
\pinlabel $x_{d+2}$ at 350 100
\pinlabel $x_{d+3}$ at 350 60
\pinlabel $x_{2d}$ at 350 -20
\pinlabel $a_1$ at 330 250
\pinlabel $a_2$ at 310 340
\pinlabel $a_3$ at 285 360
\pinlabel $a_{d}$ at 170 395
\pinlabel $C$ at 30 180
\endlabellist
\begin{center}
\includegraphics[scale=0.35]{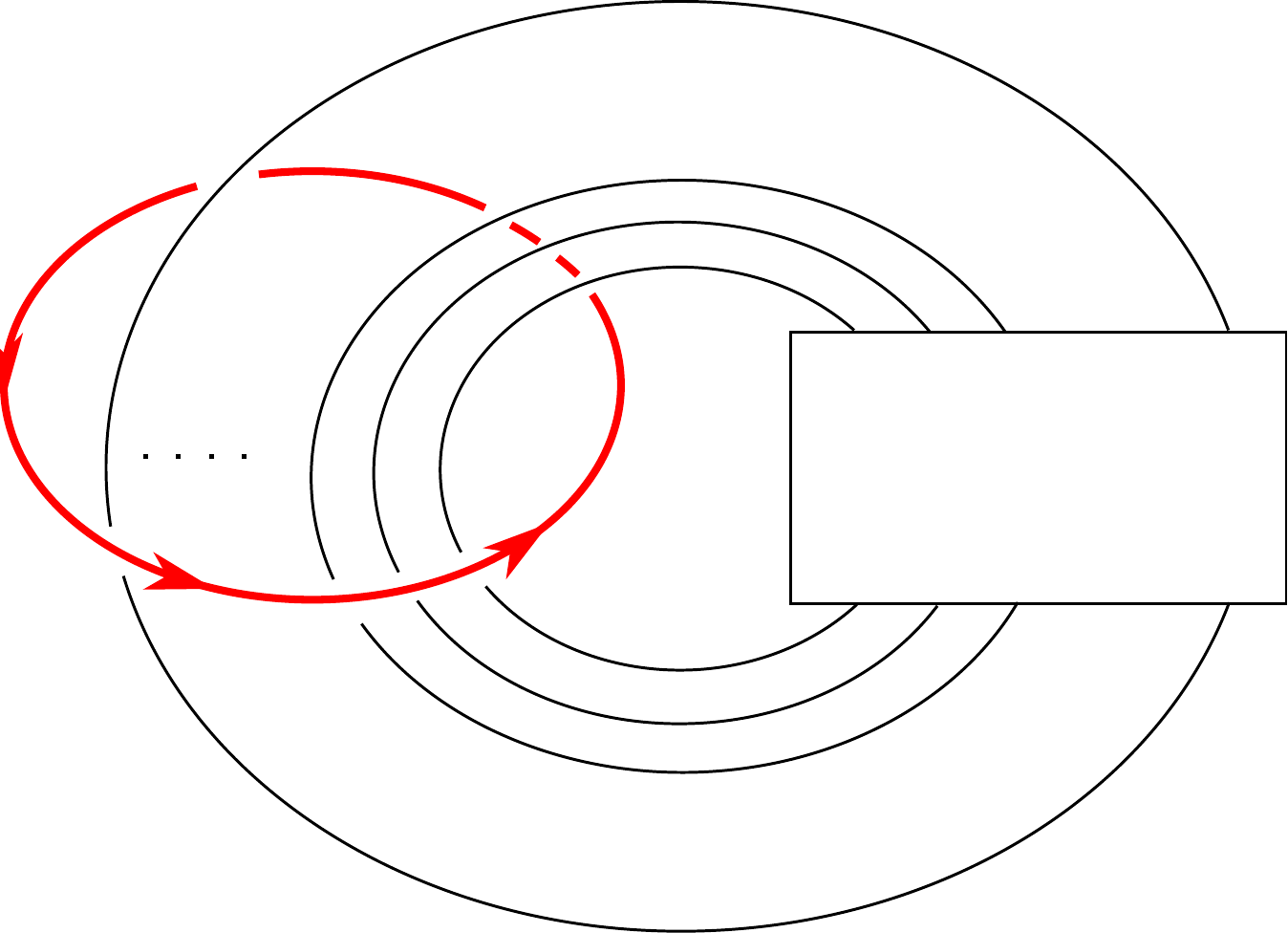}
\caption{Surgery diagram of a link $L$ in the solid torus}
\label{fig2}
\end{center}
\end{figure}
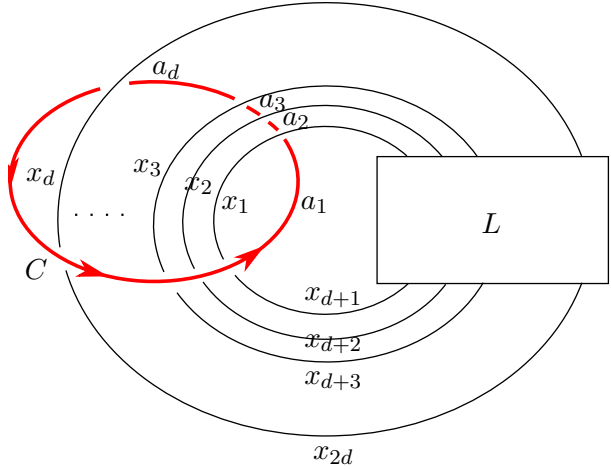
Note that our ambient manifold $S^{1}\times D^{2}$ is not closed: while the regular neighbourhood of $C$ is removed from $S^{3}$, no solid torus is glued back along the resulting boundary. The general presentation  of Theorem \ref{th1} could be adapted to the case of $M=S^{1}\times D^{2}$ just by removing relations $R$ that would correspond to the attachment of 2-handles.  So the general presentation of the fundamental quandle $Q(L)$ is given by 
\begin{xalignat*}{1}
& [ \{x_{1},\ldots ,x_n\},\{a_1,\ldots ,a_{d}\}\, \colon \, C_{1}(D_{L})\cup \{(x_{d+i})^{a_1}=x_i\textrm{ for }i=1,\ldots ,d\},\\
& \{\overline{\epsilon (x_{1})^{\delta _{1}}}a_{1}\epsilon (x_{1})^{\delta _{1}}=a_{2},\overline{\epsilon (x_{2})^{\delta _{2}}}a_{2}\epsilon (x_{2})^{\delta _{2}}=a_3,\ldots ,\overline{\epsilon (x_{d})^{\delta _{d}}}a_{d}\epsilon (x_{d})^{\delta _{d}}=a_{1)}\}]\;,
\end{xalignat*} which simplifies to a presentation with a single operator generator:
\begin{xalignat*}{1}
&\left  [\{x_{1},\ldots ,x_n\},\{a_1\}\, \colon \, C_{1}(D_{L})\cup \{(x_{d+i})^{a_1}=x_i\textrm{ for }i=1,\ldots ,d\},\{[a_1,\epsilon (x_{1})^{\delta _{1}}\ldots \epsilon (x_{d})^{\delta _{d}}]=1\}\right ] \;,
\end{xalignat*} 
 where $[a,b]$ denotes the commutator $aba^{-1}b^{-1}$.

\section{Fundamental quandle of a link in a lens space}\label{sec4}

The lens space $L_{p,q}$ is obtained from the 3-sphere by $(-p/q)$ surgery on the unknot. In other words, remove an unknotted solid torus $V$ from $S^{3}$. Denote by $\mu $ and $\lambda $ the meridian and preferred longitude of the remaining solid torus $U$, and by $\mu '$ the meridian of $V$. Now glue the removed part back by a homeomorphism $\phi \colon \partial V\to \partial U$, gluing $\mu '$ along a simple closed curve homologous to $p\lambda -q\mu$.  \\

\begin{figure}[h!]
\labellist
\normalsize \hair 2pt
\pinlabel $L$ at 520 230
\pinlabel $x_{1}$ at 245 250
\pinlabel $x_{2}$ at 205 270
\pinlabel $x_{3}$ at 150 290
\pinlabel $x_{d}$ at 35 280
\pinlabel $x_{d+1}$ at 350 150
\pinlabel $x_{d+2}$ at 350 100
\pinlabel $x_{d+3}$ at 350 60
\pinlabel $x_{2d}$ at 350 -20
\pinlabel $a_1$ at 330 250
\pinlabel $a_2$ at 310 340
\pinlabel $a_3$ at 285 360
\pinlabel $a_{d}$ at 170 395
\pinlabel $\textcolor{red}{-p/q}$ at -20 210
\endlabellist
\begin{center}
\includegraphics[scale=0.35]{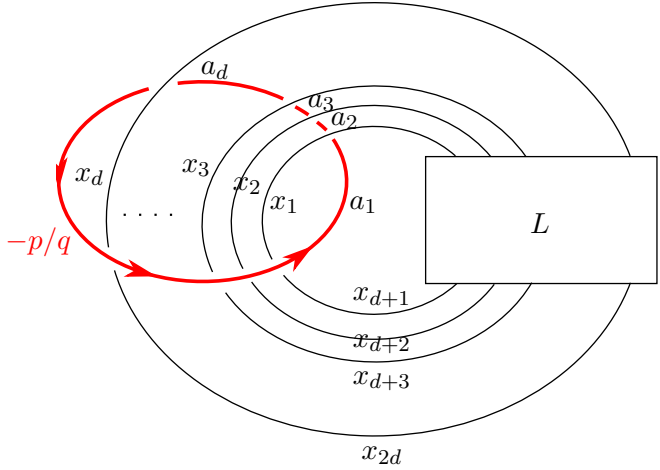}
\caption{Surgery diagram of a link $L$ in the lens space}
\label{fig3}
\end{center}
\end{figure}

Consider a link $L$ in the lens space $L_{p,q}$. In the above construction, it may be assumed that $L$ is contained inside the solid torus $U$. A surgery diagram of such link is given in Figure \ref{fig3}. The first step of the surgery (removal of the solid torus $V$) thus yields a link inside the solid torus, whose fundamental quandle was described in Section \ref{sec3} and is given by the general presentation 
\begin{xalignat*}{1}
&\left  [\{x_{1},\ldots ,x_n\},\{a_1\}\, \colon \,  C_{1}(D_{L})\cup \{(x_{d+i})^{a_1}=x_i\textrm{ for }i=1,\ldots ,d\},\{[a_1,\epsilon (x_{1})^{\delta _{1}}\ldots \epsilon (x_{d})^{\delta _{d}}]=1\}\right ] \;.
\end{xalignat*} 
In the second step of the surgery, the glueing of the 2-handle along the prescribed curve adds a new operator relation to this presentation (this relation belongs to the set $R$ in Theorem \ref{th1}). Using notation from Figure \ref{fig3}, the longitude of $U$ is given by $\lambda = a_{1}$, and the meridian of $U$ is given by $\mu =\epsilon (x_{1})^{\delta _{1}}\ldots \epsilon (x_{d})^{\delta _{d}}$. 

Therefore, the augmented fundamental quandle of the link $L$ in the lens space $L_{p,q}$ is given by the general presentation 
\begin{xalignat*}{1}
& [\{x_{1},\ldots ,x_n\},\{a_1\}\, \colon \,  C_{1}(D_{L})\cup \{(x_{d+i})^{a_1}=x_i\textrm{ for }i=1,\ldots ,d\},\\
& \left \{[a_1,\epsilon (x_{1}^{\delta _{1}})\ldots \epsilon (x_{d})^{\delta _{d}}]=1,\, a_{1}^{p}=\left (\epsilon (x_{1})^{\delta _{1}}\ldots \epsilon (x_{d})^{\delta _{d}}\right )^{q}\right \}] \;.
\end{xalignat*} 

In \cite{EH}, an analogous presentation of the fundamental rack of a link in $L_{p,1}$ was given. By counting representations of the fundamental rack of a link into a finite rack, we obtain the rack counting invariants. This idea was exploited to define rack invariants of links in $L_{p,1}$.

In  \cite{CN}, the authors associated to each link $L$ in  $L_{p,1}$ a virtual quandle invariant $VQ(L)$ that is a couple $(Q,f)$, with  $Q$ a quandle   and $f\in \aut(Q)$. The definition of $VQ(L)$ is given by describing a presentation constructed from a surgery diagram. Even if the authors find some connection with topological invariants of the link, such as homology,  the topological meaning of $VQ(L)$ is not clear: studying the connection between the above presentation of $Q(L)$  and the presentation of $VQ(L)$ could lead to a topological interpretation of $VQ(L)$.

\section{Fundamental quandle of satellite knots}\label{sec5}

A \textit{satellite knot} is a knot $S$, whose complement contains an incompressible torus that is not parallel to the boundary of a regular neighborhood of $S$. Let $S$ be a satellite knot, and denote by $T$ an incompressible, non boundary-parallel torus in the complement of $S$. Denote by $V$ the solid torus containing $S$, whose boundary is $T$. The core of $V$ is a knot $C$, that we call the \textit{companion} of $S$. Let $f\colon V\to S^{1}\times D^{2}$ be a homeomorphism that takes a preferred longitude and meridian of $V$ to the preferred longitude and meridian of the standard solid torus $U=S^{1}\times D^{2}$ in $S^{3}$. Then $P=f(S)$ is a knot in the solid torus that is called the \textit{pattern} of $S$. Moreover if we orient $C$ so that $S$ is homologous in $V$ to $wC$ with $w$ a non-negative integer number (there a choice to be made when $w=0$), we call  $w$  \textit{the winding number} of $S$.\\

Since $P$ can be considered both as a knot in $U$ and a knot in $S^3$, we will use the notation $\pi_1(P)=\pi_1(S^3-N_P)$  and $\pi_1(U-P)=\pi_1(U-N_P)$. Analogously, $Q(P)$ will denote the fundamental quandle of $P$ in $S^3$ and $Q(U-P)$ will denote the augmented fundamental quandle of $P$ in $U$. 

\begin{remark} The augmented fundamental quandle $Q(U-P)$ and the fundamental quandle $Q(P)$ of the pattern as a knot in $S^{3}$ have a straightforward connection. Namely, a primary presentation for $Q(P)$ is obtained from the general presentation of $Q(U-P)$, given at the end of Section 3, by sending the operator generator $a_1$ to the identity: $$\left[\{x_{1},\ldots ,x_{n}\}\colon C_{1}(D_L)\cup \{x_{d+i}=x_{i}\textrm{ for }i=1,\ldots ,d\}\right ]\;.$$ 
\end{remark}

It is well-known that the fundamental group of a satellite knot can be described in terms of the fundamental groups of its companion and pattern. 

\begin{proposition} \cite[Proposition 3.11]{BZ}  \label{prop1} Let S be a satellite knot with companion $C$ and pattern $P$. Let $f\colon V\to U$ be a homeomorphism between a regular neighborhood of $C$ and an unknotted solid torus $U$, such that $f(S)=P$. Denote by $i_{1}\colon \partial V\to S^{3}-C$ and $i_{2}\colon \partial U\to U-P$ the respective inclusions. Then  $\pi_1(S)\cong\pi_1(C)*_{(\pi_1(\partial V), i_1,i_2)}\pi_1(U-P)$.
\end{proposition}
More precisely, the isomorphism $\Phi: \pi_1(C)*_{(\pi_1(\partial V), i_1,i_2)}\pi_1(U-P)\to \pi_1(S)$ is defined as follows.
Choose a basepoint $z\in \partial V$ that will be omitted from the fundamental groups to simplify notation.
Define $\Phi \left(c_{1}p_{1}\ldots c_{r}p_{r}\right)= c_{1}f^{-1}(p_{1})\ldots c_{r}f^{-1}(p_{r}) $, where $c_{i}$ denotes a homotopy class of a loop in $S^{3}-C$ based at $z$, while $p_{i}$ denotes the homotopy class of a loop in $U-P$, based at $f(z)$.

 Since the fundamental quandle of any knot $K$ is a generalization of the fundamental group of $K$, it is plausible to expect that the fundamental quandle of a satellite knot may as well be described in terms of the fundamental quandles/fundamental groups of its companion and pattern knots. In what follows, we are going to specify this description. \\

Set $G_{P,C}= \pi_1(C)*_{(\pi_1(\partial V), i_1,i_2)}\pi_1(U-P)$ and denote by $\iota \colon \pi _{1}(U-P)\to G_{P,C}$ the inclusion into the amalgamated product. The fundamental group $\pi _{1}(U-P)$ admits a natural action on the fundamental quandle $Q(U-P)$ (see the definition of fundamental quandle in Section \ref{sub23}). Consider the set $Q(U-P)\times G_{P,C}$, on which we define a relation $\sim $ by $$(\alpha ,\iota (b)g)\sim (\alpha ^{b},g)\textrm{ for every }\alpha \in Q(U-P), b\in \pi _{1}(U-P)\textrm{ and }g\in G_{P,C}\;.$$  It is easy to check that $\sim $ is an equivalence relation on $Q(U-P)\times G_{P,C}$ and that the following diagram commutes 
\begin{displaymath}
\xymatrix{
G_{P,C} \ar@{->}[d]_{\Phi }  & \quad \\
\pi _{1}(S) & \pi _{1}(U-P) \ar@{->}[ul]_{\iota } \ar@{->}[l]^{(f^{-1})}}
\end{displaymath}

Using $\Phi $, the equivalence relation $\sim $ may also be expressed in another way.

\begin{lemma} \label{lemma0}For two elements $(\alpha ,g),(\beta ,h)\in Q(U-P)\times G_{P,C}$ we have $$(\alpha ,g)\sim (\beta ,h)\Leftrightarrow f^{-1}(\alpha )\Phi (gh^{-1})=f^{-1}(\beta )\;.$$
\end{lemma}
\begin{proof}
$(\Rightarrow )$ Suppose that $(\alpha ,g)\sim (\beta ,h)$; then there exists a $c\in \pi _{1}(U-P)$ so that $g=\iota (c)h$ and $\beta =\alpha ^{c}$. It follows that $f^{-1}(\alpha )\Phi (gh^{-1})=f^{-1}(\alpha )\Phi (\iota (c))=f^{-1}(\alpha )f^{-1}(c) = f^{-1}(\alpha ^{c})=f^{-1}(\beta )$.   

$(\Leftarrow )$ Let $f^{-1}(\alpha )\Phi (gh^{-1})=f^{-1}(\beta )$ for two elements $(\alpha ,g),(\beta ,h)\in Q(U-P)\times G_{P,C}$. Then the diagram above implies that $gh^{-1}\in Im (\iota )$ and therefore $f(\Phi (gh^{-1}))=\iota ^{-1}(gh^{-1})$. We have $\alpha f(\Phi (gh^{-1}))=\alpha \iota ^{-1}(gh^{-1})=\beta $ and thus 
\begin{equation*}
(\beta ,h)=(\alpha ^{\iota ^{-1}(gh^{-1})},h)\sim (\alpha ,gh^{-1}h)=(\alpha ,g)\,.\qedhere
\end{equation*}
\end{proof}

\begin{lemma} \label{lemma1} The pair $\left ((Q(U-P)\times G_{P,C})/_{\sim },G_{P,C}\right )$, equipped with the map $\epsilon \colon (Q(U-P)\times G_{P,C})/_{\sim }\to G_{P,C}$, given by $\epsilon [\alpha ,g]=\overline{g}\,\overline{\alpha }m_{\alpha (0)}\alpha \,g$, and the canonical right  action of $G_{P,C}$ on the second entry, defines an augmented quandle. 
\end{lemma}
\begin{proof} The group $G_{P,C}$ acts on the product $Q(U-P)\times G_{P,C}$ by right multiplication on the second entry, and this action descends to the quotient $(Q(U-P)\times G_{P,C})/_{\sim }$. Compute 
\begin{eqnarray*}
 [\alpha ,g]^{\epsilon [\alpha ,g]}&=&[\alpha ,g]^{\overline{g}\,\overline{\alpha }m_{\alpha (0)}\alpha \,g}=[\alpha ,\overline{\alpha }m_{\alpha (0)}\alpha g]= [\alpha *\alpha ,g]=[\alpha ,g]\\
 \epsilon ([\alpha ,g]^{h})&=&\epsilon [\alpha ,gh]=\overline{h}\,\overline{g}\,\overline{\alpha }m_{\alpha (0)}\alpha \,g\,h=h^{-1}\epsilon [\alpha ,g]h
\end{eqnarray*} to see that the map $\epsilon $ satisfies both augmentation identities from Definition \ref{def2}. 
\end{proof}


\begin{Theorem} \label{th2} The fundamental quandle of the satellite knot $S$ is isomorphic to the augmented quandle $\left ((Q(U-P)\times G_{P,C})/_{\sim },G_{P,C}\right )$. 
\end{Theorem}

\begin{proof} The fundamental quandle of the satellite has a natural structure of an augmented quandle $(Q(S),\pi _{1}(S))$. Choose a point $z\in \partial V$, that will serve as a common basepoint for $Q(S)$, $\pi _{1}(S)$ and $\pi _{1}(C)$. The action of $\pi _{1}(S)$ on $Q(S)$ is given by $\alpha ^{g}=\alpha g$ for any $\alpha \in Q(S)$ and $g\in \pi _{1}(S)$. As we recall from Subsection \ref{sub23}, the augmentation map $\epsilon _{S}\colon Q(S)\to \pi _{1}(S)$ is given by $\epsilon _{S}(\alpha )=\overline{\alpha }m_{\alpha (0)}\alpha $.

Now consider the augmented quandle $\left ((Q(U-P)\times G_{P,C})/_{\sim },G_{P,C}\right )$, obtained in Lemma \ref{lemma1}. We need to show that the augmented quandles $\left ((Q(U-P)\times G_{P,C})/_{\sim },G_{P,C}\right )$ and $(Q(S),\pi _{1}(S))$ are isomorphic (see  Definition \ref{def6}).

By Proposition \ref{prop1}, the groups $G_{P,C}$ and $\pi _{1}(S)$ are isomorphic, and, as recalled, the isomorphism $\Phi \colon G_{P,C}\to \pi _{1}(S)$ is given by $\Phi \left(c_{1}p_{1}\ldots c_{r}p_{r}\right)=c_{1}f^{-1}(p_{1})\ldots c_{r}f^{-1}(p_{r})$. Define a map $\Psi \colon (Q(U-P)\times G_{P,C})/_{\sim }\to Q(S)$ by $$\Psi \left [\alpha ,g\right ]=f^{-1}(\alpha )^{\Phi \left (g\right )}=f^{-1}(\alpha)\Phi \left (g\right )\;.$$ 
First observe that for any $\alpha \in Q(U-P)$, $b \in \pi _{1}(U-P)$ and $g\in G_{P,C}$ we have 
$$\Psi [\alpha ,\iota(b)g]=f^{-1}(\alpha )\Phi (\iota(b)g)=f^{-1}(\alpha )f^{-1}(b)\Phi (g)=f^{-1}(\alpha ^{b})\Phi (g)=\Psi [\alpha ^{b},g]\;,$$ thus $\Psi $ is well defined. It follows from Lemma \ref{lemma0} that $\Psi $ is injective. 

To show that $\Psi $ is surjective, choose an element $s\in Q(S)$. Recall that $s$ is the homotopy class of a path $\gamma \colon I\to E_{S}$ from $\gamma (0)\in \partial N_{S}$ to $\gamma (1)=z$. We may assume that the trajectory of $\gamma $ intersects the torus $\partial V$ transversely in a finite set of points. Define 
$$t_{0}=  \begin{cases}
& \min \,\{t\in I\colon \gamma (t)\in \partial V\}\;, \textrm{ if }\gamma (I)\cap \partial V\neq \emptyset\\
& 1\;, \textrm{ otherwise}.
\end{cases} $$ Choose a path $\beta \colon I\to \partial V$ from $\gamma (t_{0})$ to $z$. Then $\gamma \simeq \gamma |_{[0,t_{0}]}\,\beta \, \overline{\beta }\, \gamma |_{[t_{0},1]}$. Note that $\alpha :=\gamma |_{[0,t_{0}]}\, \beta $ is a path in the $\textup{closure}(V-N_{S})$ from $\partial N_{S}$ to the basepoint $z$, while $g:=\overline{\beta }\, \gamma |_{[t_{0},1]}$ is a loop in $E_{S}$, based at $z$. It follows that $\gamma \simeq f^{-1}(f(\alpha ))\Phi (\Phi ^{-1}(g))$ and thus $\gamma =\Psi [f(\alpha ),\Phi ^{-1}(g)]$.

Consider the diagram 
\begin{displaymath}
\xymatrix{
\left (Q(U-P)\times G_{P,C}\right )/_{\sim }\times G_{P,C} \ar@{->}[d]^{\Psi \times \Phi } \ar@{->}[r] & \left (Q(U-P)\times G_{P,C} \right )/_{\sim }\ar@{->}[r]^{\quad \epsilon} \ar@{->}[d]^{\Psi } & G_{P,C} \ar@{->}[d]^{\Phi }\\
Q(S)\times \pi _{1}(S) \ar@{->}[r] & Q(S) \ar@{->}[r]^{\epsilon _{S}} & \pi _{1}(S)}
\end{displaymath}
 The diagram above is commutative, indeed we have
\begin{eqnarray*}
 \Psi \left (\left [\alpha ,g\right ]^{\epsilon \left [\beta ,h\right ]}\right )&=&\Psi \left [\alpha ,g\epsilon (\beta ,h)]\right ]=\Psi \left [\alpha ,g\overline{h}\,\overline{\beta }m_{\beta (0)}\beta h\right ]\\ 
 (\Psi \times \Phi )[\alpha ,g]^{\epsilon [\beta ,h]}&=& \left(f^{-1} (\alpha )\Phi (g)\right )^{\Phi (\epsilon [\beta ,h])}=f^{-1}(\alpha)\Phi( g\overline{h}\,\overline{\beta }m_{\beta (0)}\beta h)\\
 \Phi (\epsilon (\alpha ,g))&=&\Phi \left (\overline{g}\,\overline{\alpha }m_{\alpha (0)}\alpha g\right )\\
 \epsilon _{S}\left (\Psi [\alpha ,g]\right )&=&\epsilon _{S}\left (f^{-1}(\alpha)\Phi (g)\right )=\Phi \left (\overline{g}\,\overline{\alpha }m_{\alpha (0)}\alpha g\right )
\end{eqnarray*}
for every $\left [\alpha ,g\right ]$ and $\left [\beta ,h\right ]$ in $(Q(U-P)\times G_{P,C})/_{\sim }$.
 
 We have shown that $\Phi $ and $\Psi $ are both bijections, thus they define an isomorphism of augmented quandles. 
\end{proof}

By Theorem \ref{th2}, a presentation for the fundamental quandle of a satellite knot may be recovered from the presentations of fundamental quandles of companion and pattern knots. Note that the pattern knot is a knot in the solid torus, whose fundamental quandle has been described in Section \ref{sec3}. 

\begin{corollary} \label{cor_pres} Let $S$ be a satellite knot with companion $C$ and pattern $P$. Let $f\colon V\to U$ be a homeomorphism between a regular neighbourhood of $C$ and an unknotted solid torus $U$, such that $f(S)=P$. Denote by $\mu _{V}$ and $\lambda _{V}$ (respectively $\mu _{U}$ and $\lambda _{U}$) the meridian and preferred longitude of the solid torus $V$ (respectively $U$). Given a general presentation $[S_{P1},S_{P2}\colon R_{P1},R_{P2}]$ of the augmented fundamental quandle $Q(U-P)$ and a presentation $\langle S_{C}\colon R_{C}\rangle $ of the fundamental group $\pi_1(C)$, the augmented fundamental quandle $Q(S)$ has a general presentation 
\begin{xalignat}{1} \label{present} 
& \left [S_{P1}, S_{P2}\cup S_{C}\, \colon \, R_{P1},R_{P2}\cup R_{C}\cup \{i_{1}(\mu _{V})=i_{2}(\mu _{U}), i_{1}(\lambda _{V})=i_{2}(\lambda _{U})\}\right ]\;,
\end{xalignat}
where $i_{1}\colon \partial V\to S^{3}-C$ and $i_{2}\colon \partial U\to U-P$ are respective inclusions.
\end{corollary} 

\begin{remark}\label{standard meridians and longitudes}
 When using the general presentation \eqref{present} to find the fundamental quandle of a satellite knot based on a knot diagram, meridians and longitudes of the solid tori $U$ and $V$ may be chosen in a standard way. Namely, in the diagram of the companion knot, the generator of $\pi _{1}(C)$ corresponding to any arc may be chosen as the meridian $\mu _{V}$. The longitude $\lambda _{V}$ is obtained as the product of generators corresponding to arcs one passes under when going around the companion knot once, taking each term of the product to the power that equals the sign of the crossing. The longitude $\lambda _{U}$ corresponds to the ``red'' arc in the diagram of the pattern (this is the arc $a_{1}$ in Figure \ref{fig2}). The meridian $\mu _{U}$ corresponds to the product $\epsilon (x_{1})^{\delta _{1}}\epsilon (x_{2})^{\delta _{2}}\ldots \epsilon (x_{d})^{\delta _{d}}$ in Figure \ref{fig2}, where $w=\sum _{j=1}^{d}\delta_j$ is the winding number of the knot.
\end{remark}

\begin{figure}[h!]
\begin{center}
\includegraphics[scale=0.6]{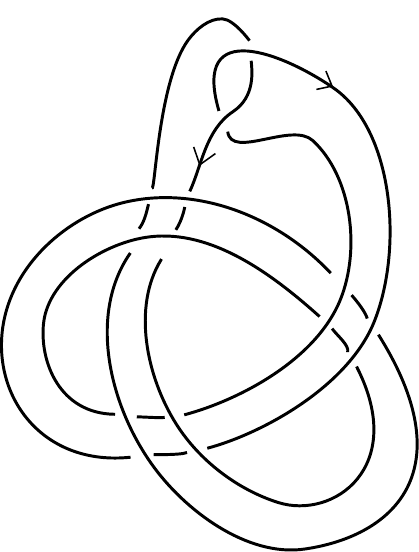}
\caption{The double of the trefoil knot}
\label{fig4}
\end{center}
\end{figure}

\begin{example}[The double of the trefoil knot] \label{ex1} Consider the simple satellite in Figure \ref{fig4}, whose companion and pattern knots are given by Figure \ref{fig5}. The presentation of the fundamental quandle of $P$ is given by $$Q(U-P)=\left [\{x_{1},x_{2},x_{3},x_{4}\}, \{a\}\, \colon \{x_{1}*x_{4}=x_{2}, x_{4}*x_{1}=x_{3}, x_{3}^{a}=x_{1}, x_{4}^{a}=x_{2}\}, \{[a,\overline{\epsilon (x_{1})}\epsilon (x_{2})]=1\}\right ]\;,$$ while the fundamental group of the trefoil is given by the presentation 
\begin{align*}
\pi _{1}(C)&=\langle\{y_{1}, y_{2},y_{3}\, \}\, \colon \{\overline{y}_{3}y_{1}y_{3}=y_{2}, \overline{y}_{2} y_{3} y_{2}=y_{1}, \overline{y}_{1} y_{2} y_{1}=y_{3}\}\rangle. 
\end{align*} We may assume that $i_{1}(\mu _{V})=y_{3}$, $i_{1}(\lambda _{V})=y_{3}y_{1}y_{2}$, $i_{2}(\mu _{U})=\overline{\epsilon (x_{1})}\epsilon (x_{2})$ and $i_{2}(\lambda _{U})=a$. It follows that the fundamental quandle of the satellite is given by the presentation 
\begin{xalignat*}{1}
& Q(S)=[\{x_{1},x_{2},x_{3},x_{4}\},\, \{a, y_{1}, y_{2},y_{3}\, \}\, \colon  \{x_{1}*x_{4}=x_{2}, x_{4}*x_{1}=x_{3}, x_{3}^{a}=x_{1}, x_{4}^{a}=x_{2}\},\\ 
& \{[a,\overline{\epsilon (x_{1})}\epsilon (x_{2})]=1, \overline{y}_{3}y_{1}y_{3}=y_{2}, \overline{y}_{2}y_{3}y_{2}=y_{1}, \overline{y}_{1}y_{2}y_{1}=y_{3}, y_{3}=\overline{\epsilon (x_{1})}\epsilon (x_{2}), a=y_{3}y_{1}y_{2} \}] 
\end{xalignat*} that simplifies to 
\begin{xalignat*}{1}
& Q(S)=[\{x_{1},x_{4}\},\, \{y_{1}, y_{2}\, \}\, \colon  \{(x_{4}*x_{1})^{\overline{y}_{1}y_{2}y_{1}^{2}y_{2}}=x_{1}, x_{4}^{\overline{y}_{1}y_{2}y_{1}^{2}y_{2}}=x_{1}*x_{4}\},\\ 
& \{y_{1}y_{2}y_{1}=y_{2}y_{1}y_{2}, y_{1}^{2}y_{2}^{2}y_{1}=y_{2}y_{1}^{2}y_{2}^{2}, \overline{y}_{1}y_{2}y_{1}=\overline{\epsilon (x_{1})}\epsilon (x_{1}*x_{4}) \}] \;.
\end{xalignat*}
\begin{figure}[h!]
\labellist
\normalsize \hair 2pt 
\pinlabel $y_3$ at 120 232
\pinlabel $y_1$ at 160 100
\pinlabel $y_2$ at -10 150
\pinlabel $x_{2}$ at 400 110 
\pinlabel $x_{1}$ at 465 110
\pinlabel $x_{4}$ at 300 110
\pinlabel $x_{3}$ at 240 110
\pinlabel $a$ at 370 8
\pinlabel $P$ at 450 210
\pinlabel $C$ at 20 210
\endlabellist
\begin{center}
\includegraphics[scale=0.6]{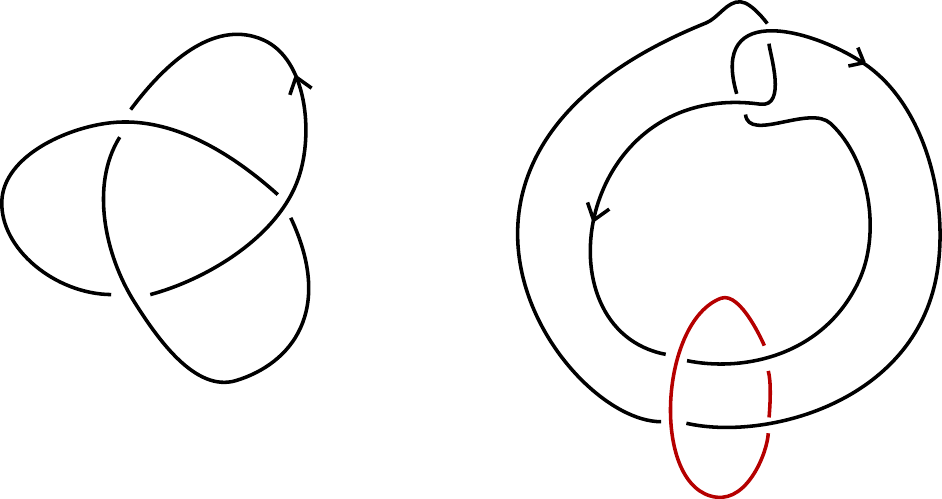}
\caption{Companion and pattern of the knot from Example \ref{ex1}}
\label{fig5}
\end{center}
\end{figure}
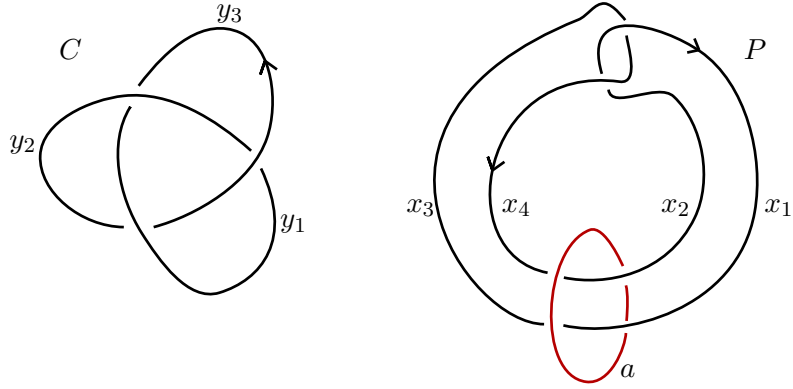
\end{example}

\begin{example}[$(5,2)$-cable of the trefoil knot]  A \textit{cable knot} is a satellite knot whose pattern is a torus knot. For coprime integers $p$ and $q$, the torus knot of type $(p,q)$ is the knot $T(p,q)$, that is ambient isotopic to the simple closed curve $p\mu_U+q\lambda_U$ on $\partial U$. The satelllite knot with companion $C$ and pattern $T(p,q)$ is called a $(p,q)$-cable of $C$. 
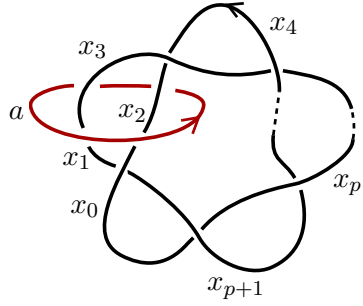
\begin{figure}[h!]
\labellist
\normalsize \hair 2pt
\pinlabel $x_{3}$ at 40 140
\pinlabel $x_{0}$ at 35 40
\pinlabel $x_{p+1}$ at 130 -10
\pinlabel $x_{p}$ at 200 55
\pinlabel $x_{4}$ at 160 155
\pinlabel $a$ at -8 100
\pinlabel $x_{1}$ at 30 70
\pinlabel $x_{2}$ at 65 100
\endlabellist
\begin{center}
\includegraphics[scale=0.6]{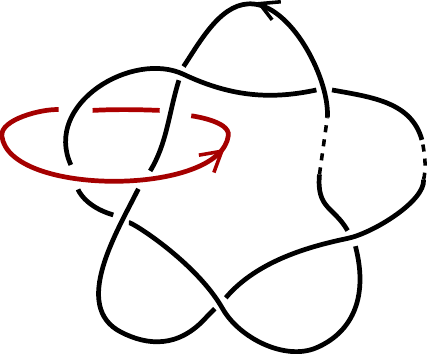}
\caption{Diagram of the torus knot $T(p,2)$ inside the solid torus $U$.}
\label{fig6}
\end{center}
\end{figure}

Consider the torus knot $T(p,2)$ lying on the boundary of the solid torus $U$, whose diagram is shown in Figure \ref{fig6}. As a knot in a solid torus, its fundamental quandle admits a presentation
\begin{xalignat*}{1}
& Q(U-P) = [\{x_{0},x_{1},\ldots ,x_{p+1}\}, \{a\}\, \colon \\
& \{x_{1}\ast x_{0}=x_{p+1},x_{0}\ast x_{p+1}=x_{p},\ldots ,x_{4}\ast x_{3}=x_{2},x_{1}^{a}=x_{3}, x_{0}^{a}= x_{2}\}, \{[a,\overline{\epsilon (x_{2}})\,\overline{\epsilon (x_{3})}]=1\}] \;,
\end{xalignat*} that simplifies to 
\begin{xalignat*}{1}
& Q(U-P) = [\{x_{0},x_{1}\}, \{a\}\, \colon \{x_{1}^{a}=\overbrace{x_{0}\ast x_{1}\ast \ldots \ast x_{0}}^{p},\, x_{0}^{a}=\overbrace{x_{1}\ast x_{0}\ast \ldots \ast x_{0}}^{p+1}\},\\
& \{[\epsilon (\underbrace{x_{0}\ast x_{1}\ast \ldots \ast x_{0}}_{p})\epsilon (\underbrace{x_{1}\ast x_{0}\ast \ldots \ast x_{0}}_{p+1}),a]=1\}] \;.
\end{xalignat*}  The fundamental group of the trefoil knot has a presentation
\begin{align*}
\pi _{1}(C)&=\langle\{y_{1}, y_{2}, y_{3}\, \}\, \colon \{\overline{y}_{3}y_{1}y_{3}=y_{2}, \overline{y}_{2} y_{3} y_{2}=y_{1}, \overline{y}_{1} y_{2} y_{1}=y_{3}\}\rangle. 
\end{align*} We may assume that $i_{1}(\mu _{V})=y_{3}$, $i_{1}(\lambda _{V})=y_{3}y_{1}y_{2}$, $i_{2}(\mu _{U})=\overline{\epsilon (x_{2})}\, \overline{\epsilon (x_{3})}$ and $i_{2}(\lambda _{U})=a$. 
Using Corollary \ref{cor_pres}, it is possible to obtain 
a presentation for the $(p,2)$ cable of the trefoil knot. Just as an example, we work out computations in the case $p=5$:
\begin{xalignat*}{1}
& Q(S)=[\{x_{0},x_{1}\},\, \{a, y_{1}, y_{2},y_{3}\, \}\colon \{x_{1}^{a}=x_{0}\ast x_{1}\ast x_{0}\ast x_{1}\ast x_{0},\, x_{0}^{a}=x_{1}\ast x_{0}\ast x_{1}\ast x_{0}\ast x_{1}\ast x_{0}\},\\
& \{[\epsilon (x_{0}\ast x_{1}\ast x_{0}\ast x_{1}\ast x_{0})\epsilon (x_{1}\ast x_{0}\ast x_{1}\ast x_{0}\ast x_{1}\ast x_{0}),a]=1,\overline{y}_{3}y_{1}y_{3}=y_{2}, \overline{y}_{2}y_{3}y_{2}=y_{1}, \\
& \overline{y}_{1}y_{2}y_{1}=y_{3}, y_{3}=\overline{\epsilon (x_{1}\ast x_{0}\ast x_{1}\ast x_{0}\ast x_{1}\ast x_{0})}\,\overline{\epsilon (x_{0}\ast x_{1}\ast x_{0}\ast x_{1}\ast x_{0}}), a=y_{3}y_{1}y_{2} \}] 
\end{xalignat*} that simplifies to 
\begin{xalignat*}{1}
& Q(S)=[\{x_{0},x_{1}\},\, \{y_{1}, y_{2}\, \}\colon \\
& \{x_{1}^{\overline{y}_{1}y_{2}y_{1}^{2}y_{2}}=x_{0}\ast x_{1}\ast x_{0}\ast x_{1}\ast x_{0}, x_{0}^{\overline{y}_{1}y_{2}y_{1}^{2}y_{2}}=x_{1}\ast x_{0}\ast x_{1}\ast x_{0}\ast x_{1}\ast x_{0}\},\\
& \{y_{1}y_{2}y_{1}=y_{2}y_{1}y_{2}, \overline{y}_{1}y_{2}y_{1}=\overline{\epsilon (x_{1}\ast x_{0}\ast x_{1}\ast x_{0}\ast x_{1}\ast x_{0})}\,\overline{\epsilon (x_{0}\ast x_{1}\ast x_{0}\ast x_{1}\ast x_{0}}),\\
& y_{1}y_{2}\overline{y}_{1}y_{2}y_{1}=y_{2}y_{1}\overline{y}_{2}y_{1}y_{2}\}]\;.
\end{xalignat*}
\end{example}

Theorem \ref{th2} describes a decomposition of the fundamental quandle of a satellite knot along the incompressible torus $\partial V$ that is the boundary of a regular neighbourhood of the companion knot. 
The quandle of a satellite knot $S$ may also be decomposed into simpler pieces in a different way. By choosing a suitable 2-sphere that intersects $S$ in a finite number of points, the satellite knot is expressed as a closure of the tensor product of two tangles: one of them belongs to the pattern in the other one to companion knot. By studying the fundamental quandle functor defined in \cite{CH}, the fundamental quandle of a satellite knot was expressed in terms of quandles of tangles in its decomposition.

 



\section{The Alexander module of a connected quandle} \label{sec6}

	\subsection{Abelian and affine quandles} \label{sub61}


In \cite{CP} the authors developed a commutator theory for quandles in the sense of Freese and McKenzie \cite{comm}. Such theory defines the notions of abelianness and centrality for congruences of arbitrary algebraic structures and consequently also the notions of abelian, nilpotent and solvable algebraic structures. For racks and quandles, such properties are completely reflected by the group theoretical properties of the displacement group and its subgroups. For sake of simplicity, we use the characterization of abelian quandles obtained in \cite{CP} as the definition: a quandle $Q$ is said to be {\it abelian} if $\dis(Q)$ is abelian and semiregular (the pointwise stabilizers are trivial).
%

Let us denote by $\gamma_Q$ the smallest congruence of $Q$ with abelian quotient (such congruence can be defined by any algebraic structure). The quotient $Q/\gamma_Q$ can be also defined by the following universal property: if there exists a surjective quandle homomorphism $h:Q\longrightarrow Q^\prime$ and $Q^\prime$ is abelian, then $h$ factors through $Q/\gamma_Q$. 

\begin{lemma}\label{char}
Let $Q$ be a quandle, $N$ be a normal subgroup of $\aut(Q)$ and let $\sim_N$ be the orbit decomposition with respect to $N$. Then $\sim_N$ is a congruence of $Q$ and
$$\pi_N:\aut(Q)\longrightarrow \aut(Q/\sim_N),\quad h\mapsto \pi_N(h)$$
defined by setting $\pi_N(h)([a])=[h(a)]$ is a well defined group homomorphism.
\end{lemma}

\begin{proof}
Let $a \sim_N c$ and $b\sim_N d$, i.e. $c=n_1(a)$ and $d=n_2(b)$ for some $n_1,n_2\in N$. Then
$$R_c^{\pm 1}(d)=R_{n_1(a)}^{\pm 1}n_2(b)=\underbrace{n_1 R_{a}^{\pm 1}n_1^{-1}n_2 R_a^{\mp 1}}_{\in N} R_a^{\pm 1}(b).$$
Therefore $R_c^{\pm 1}(d) \sim_N R_a^{\pm 1}(b)$ and so $\sim_N$ is a congruence. 

 For the latter statement, we need to prove that if $a \sim_N b$, then $h(a) \sim_N h(b)$ for every $h\in \aut(Q)$. Let $b=n(a)$ for $n\in N$, we have
$$h(b)=h(n(a))=hnh^{-1}h(a)$$
and since $hnh^{-1}\in N$, then $h(a)\sim_N h(b)$. 
\end{proof}

 According to \cite{GB}, if $Q$ is a connected quandle, then $\gamma_Q$ is the orbit decomposition with respect to the derived subgroup of $\dis(Q)$. In such case, we can apply Lemma \ref{char}.

\begin{corollary}\label{homo}
Let $Q$ be a connected quandle. Then 
$$\pi_{\gamma}:\aut(Q)\longrightarrow \aut(Q/\gamma_Q),\quad h\mapsto \pi_\gamma(h)$$
defined by setting $\pi_{\gamma}(h)([a])=[h(a)]$, is a well defined group homomorphism.
\end{corollary}
%

Let $A$ be an abelian group and $f\in \aut(A)$. The algebraic structure $(A,*)$, where
$$x*y=f(x)+(1-f)(y)$$	
for every $x,y\in A$ is a quandle, denoted by $\aff(A,f)$ and we call it {\it affine} or {\it Alexander} quandle. Affine quandles are examples of abelian quandles. On the other hand, connected abelian quandles are affine \cite{hsv}.

Affine quandles and modules over the ring of Laurent polynomials $\Lambda=\mathbb{Z}[t,t^{-1}]$ are essentially the same thing. Indeed, every affine quandle $\aff(A,f)$ is a $\Lambda$-module, defined by $t^{\pm 1} \cdot x=f^{\pm 1}(x)$ for every $x\in A$. On the other hand, if $M$ is a $\Lambda$-module, then $\aff(M,t)$ is an affine quandle. Such correspondence between affine quandles and $\Lambda$-modules has been very effective to solve quandle theoretical problems, see for instance \cite{Hou}.

In this setting, the orbits with respect to the action of $\dis(\aff(M,t))$ are the cosets with respect to the submodule $(1-t)M$, indeed
\begin{equation}\label{dis of affine}
R_x R_y^{-1}(z)=(1-t)(x-y)+z,
\end{equation}
for every $x,y\in M$.

The following is a particular case of \cite[Theorem 3.14]{homQ}.
\begin{proposition}\label{hom of affine}
Let $Q_i=\aff(A_i,f_i)$ for $i=1,2$ be connected affine quandles and $M_i$ the associated $\Lambda$-modules for $i=1,2$. The following are equivalent:
\begin{itemize}
\item[(i)] $h:Q_1	\longrightarrow Q_2$ is a quandle homomorphism, such that $h(0)=b$. 
\item[(ii)] $h(x)=g(x)+b$, where $g$ is a morphism of $\Lambda$-modules between $M_1$ and $M_2$.
\end{itemize}	
\end{proposition}
%
%
%
%
%

\subsection{Presentation of the Alexander module}
\label{sub62}
 
 In this section we will denote by $\Gamma_n(G)$ the $n$th element of the derived series of a group $G$. Let $Q$ be a quandle, we define the {\it adjoint group} or {\it associated group} of $Q$ as 
\begin{equation}\label{adj}
\mathrm{Adj}(Q)=\langle Q\, :\,  \{xyx^{-1}=y*x, \, x,y\in Q\}\rangle.
\end{equation}  If $Q$ is connected, we can choose $x_1\in Q$ and we get that the adjoint group of $Q$ has the following structure
\begin{equation}\label{structure_adj}
\mathrm{Adj}(Q)\cong \Gamma_1(\mathrm{Adj}(Q))\rtimes \langle x_1	\rangle\cong \Gamma_1(\mathrm{Adj}(Q))\rtimes \Z \;,
\end{equation} 
where we can describe the set of generators of $\mathrm{Adj}(Q)$ as $E(Q)=\setof{(e_x,x_1)}{x\in Q}$ for some elements $e_x\in \Gamma_1(\mathrm{Adj}(Q))$, and in particular we may assume that $e_{x_1}=1$. The semidirect product is defined by the action of conjugation by the element $x_1$, restricted to $\Gamma_1(\mathrm{Adj}(Q))$, where $\mathbb Z$ is the multiplicative group generated by $x_1$. We denote such action by $\rho$.

Let $\mathcal{C}=\textup{Conj}(\setof{(h,x_1)}{h\in \Gamma_1(\mathrm{Adj}(Q))})$, and let $\mathcal{C}_0=\textup{Conj}(E(Q))\subseteq \mathcal{C}$. The map
$$\phi:Q\longrightarrow \mathcal{C}_0,\quad x\mapsto (e_x,x_1)$$
for $x\in Q$ is a surjective quandle homomorphism. It is easy to prove that $\Gamma_1(\mathrm{Adj}(Q))$ is generated by $\setof{(e_y,x_1)(1,x_1)^{-1}=(e_y,1)}{y\in Q}$ and since the quandle $\mathcal{C}_0$ is connected, we have that
\begin{equation}\label{the h_x}
(e_x,1)=(e_x,x_1)(1,x_1)^{-1}=(h_x,1)(1,x_1)(h_x,1)^{-1}(1,x_1)^{-1}=(h_x \rho(h_x)^{-1},1)
\end{equation}
for some $h_x\in \Gamma_1(\mathrm{Adj}(Q))$ for every $x\in Q$. 
 
The abelian group $\Gamma_1(\mathrm{Adj}(Q))/\Gamma_2(\mathrm{Adj}(Q))$ has the structure of a $\Lambda$-module, defined by the map induced by $\rho$ on $\Gamma_1(\mathrm{Adj}(Q))/\Gamma_2(\mathrm{Adj}(Q))$, i.e.
$$t\cdot \xbar{h}=\xbar{\rho(h)}$$
for every $h\in \Gamma_1(\mathrm{Adj}(Q))$, where $\xbar{h}=h+\Gamma_2(\mathrm{Adj}(Q))$ for $h\in  \mathrm{Adj}(Q)$.  Let us denote this $\Lambda $-module by $\mathcal{M}(Q)$, and we call it the {\it Alexander module of $Q$}. The {\it Alexander quandle of $Q$} is the affine quandle, associated to $\mathcal{M}(Q)$, and we denote it by $\mathrm{Alex}(Q)=\aff(\mathcal{M}(Q),t)$.

\begin{lemma}\label{Alex is a factor of QK}
Let $Q$ be a connected quandle. The map
$$\gamma:Q\longrightarrow \mathrm{Alex}(Q),\quad x\mapsto \xbar{e_x },$$
is a surjective quandle homomorphism.
\end{lemma}

\begin{proof}
Let us denote by $\xbar{h}$ the image of $h\in \mathrm{Adj}(Q)$ under the canonical map $\mathrm{Adj}(Q)\longrightarrow \mathrm{Adj}(Q)/\Gamma_2(\mathrm{Adj}(Q))$. The map
$$\mathcal{C}\longrightarrow \mathrm{Alex}(Q), \quad (h,x_1)\mapsto \xbar h$$
is a quandle homomorphism. The image of the restriction of this mapping to $\mathcal{C}_0$ is the subquandle $U=\setof{\xbar {e_x}}{x\in Q}$. The subquandle $U$ is connected and it contains $0$, therefore according to \cite[Lemma 3.7]{Principal}, $U$ is affine over a submodule of $\mathcal{M}(Q)$. 
According to \eqref{the h_x}, the group $\Gamma_1(\mathrm{Adj}(Q))/\Gamma_2(\mathrm{Adj}(Q))$ is generated by
\begin{equation*}\label{generators of the module}
\xbar{e_x} = (1-t)\xbar{h_x}
\end{equation*} for $x\in Q$ and so $(1-t)\mathcal{M}(Q)=\mathcal{M}(Q)$. Since 
$$\xbar{e_{x_1}}*\xbar{e_x}=(1-t)\xbar{e_x}\in U$$
for every $x\in Q$, then $(1-t)\mathcal{M}(Q)=\mathcal{M}(Q)\leq U$. Thus, the mapping $\gamma$ in the statement is onto.
%
\end{proof}


Let $Q=[X\colon R]$ be a presentation of the quandle $Q$. Every relation in $R$ is equivalent to one relation of this form. 
\begin{equation}\label{relation}\
 R_{x_1}^{k_1}	\ldots R_{x_n}^{k_n}(y)=z\tag{R}
\end{equation}
for some $\setof{y,z,x_i}{i=1,\ldots,n}\subseteq X$. 

 The group $\mathrm{Adj}(Q)$ is generated by the set $X$ and the relation \eqref{relation} implies that
\begin{equation}\label{conj_relation}
 e_{x_1}^{k_1}\ldots e_{x_n}^{k_n}e_y e_{x_n}^{-k_n} \ldots e_{x_1}^{-k_1}=e_z\tag{CR}
\end{equation}
holds in $\mathrm{Adj}(Q)$. In particular, $\mathrm{Adj}(Q)$ is generated by the set $X$ modulo the relations obtained from $R$ as in \eqref{conj_relation}. 

Accordingly, in $ \mathrm{Alex}(Q)$ we have that the {\it linearized relation} 
\begin{equation}\label{linearized relation}
(1-t^{k_1})\xbar{e_{x_1}}+t^{k_1}(1-t^{k_2})\xbar{e_{x_2}}+\ldots +t^{k_1+\ldots +k_{n-1}}(1-t^{k_n})\xbar{e_{x_n}}+t^{k_1+\ldots + k_n}\xbar{e_y}=\xbar{e_z}\tag{LR}
\end{equation} 
holds. We can obtain a presentation of $\mathcal M (Q)$ using the corresponding generators and the linearized relations as in \eqref{linearized relation}.

\begin{lemma}\label{presentation of M(Q)}
Let $Q=[X\colon R]$ be a presentation of a connected quandle and let $x_1\in X$. The Alexander module $\mathcal M (Q)$ is the free $\Lambda$-module, generated by the set $\setof{\xbar{e_x}}{x\in X\setminus \{x_1\}}$ modulo the relations, obtained from $R$ as in \eqref{linearized relation}.
\end{lemma}
%

\begin{proof}
According to Lemma \ref{Alex is a factor of QK}, the map $x\mapsto \xbar{e_x}$ is a surjective quandle homomorphism, so $\mathrm{Alex}(Q)$ is generated by $\setof{\xbar{e_x}}{x\in X}$, and so $\mathcal{M}(Q)$ is generated as a module by $\setof{\xbar{e_x}}{x\in X\setminus\{x_1\}}$, and the relations as in \eqref{linearized relations statement} hold. 
We need to show that this set of relations provides a presentation of $\mathcal M (Q)$.

Let $N$ be the $\Lambda$-module, generated by the set $X\setminus \{x_1\}$ and satisfying the relations as in \eqref{linearized relations statement}, and define the group $G=N\rtimes \Z$, where the action by $(0,1)$ is the action of $t$ on $N$. It is easy to verify that the generators of $G$ are $\setof{(x,1)}{x\in X\setminus \{x_1\}}\cup \{(0,1)\}$ and such generators satisfy the same relations that define the group $\mathrm{Adj}(Q)$. Hence we have a canonical map from $\mathrm{Adj}(Q)$ onto $G$ and the following commutative diagram
\begin{equation*}
\xymatrixcolsep{63pt}\xymatrixrowsep{30pt}\xymatrix{ \mathrm{Adj}(Q)\ar[r]
\ar[d] &  G  \\
\mathrm{Adj}(Q)/\Gamma_2(\mathrm{Adj}(Q))\ar[ur]^h  &  },
\end{equation*} 
where the map $h$ restricts and corestricts to $\mathcal M (Q)$ and $N$. Hence if the generators of $\mathcal M (Q)$ satisfy a relation, then also the generators of $N$ do and therefore such relation is a consequence of the set of relations of the relation obtained from $R$ as in \eqref{linearized relation}. 
\end{proof}

In the following Proposition, we show that $\Alex(Q)$ satisfies the same universal property as $Q/\gamma_Q$. This result was implicit in the main result of \cite{Inoue}.

\begin{proposition}\label{the Alex}
Let $Q$ be a connected quandle. Then $\mathrm{Alex}(Q)\cong Q/\gamma_{Q}$.
\end{proposition}

\begin{proof}
In order to prove that $\mathrm{Alex}(Q)$ and $Q/\gamma_{Q}$ are isomorphic, it is enough to show that they have the same universal property. Let $R$ be an abelian quandle and let $h:Q	\longrightarrow R$ be a surjective quandle homomorphism. First, let us assume that $h(x_1)=0$. 

If $Q$ is generated by $X$, then according to Lemma \ref{presentation of M(Q)} the Alexander module $\mathcal M (Q)$ is generated by $\setof{e_x+\Gamma_2(\mathrm{Adj}(Q))}{x\in X\setminus\{x_1\}}$ modulo the relations as in \eqref{linearized relations statement}.

The quandle $R$ is a connected affine quandle, and so $R$ has a structure of a $\Lambda$-module, generated by $\setof{h(x)}{x\in X\setminus \{x_1\}}$. So, the generators of the $\Lambda$-module, associated to $R$, satisfy the same relations as the set of generators of $\mathcal M (Q)$. Therefore, there exists a module homomorphism $\hat{h}:\mathcal{M}(Q)\longrightarrow R$, mapping $e_x+\Gamma_2(\mathrm{Adj}(Q)$ to $ h(x)$ for $x\in X\setminus \{x_1\}$. Thus, the following diagram is commutative:
\begin{equation}
\xymatrixcolsep{63pt}\xymatrixrowsep{30pt}\xymatrix{ Q\ar[r]^{\gamma}
\ar[dr]^{ h} &  \mathrm{Alex}(Q) \ar[d]^{\widehat h} \\
& R }.\label{gamma}
\end{equation} 
By Proposition \ref{hom of affine}, $\widehat{h}$ is also a quandle homomorphism and so the map $h$ factors through $\mathrm{Alex}(Q)$. 

If $h(x_1)=a$, we can consider the map $h_a(x)=h(x)-a$ with the property $h_a(x_1)=0$. Therefore $h(x)=h_a(x)+a=\widehat{h_a}\gamma(x)+a$, and so again $h$ factors through $\gamma$.
\end{proof}
%



According to \cite[Corollary 3.3]{FR}, the adjoint group of $Q(K)$, the fundamental quandle of a knot $K$ in the 3-sphere (or more generally in a homotopy 3-sphere) is the fundamental group of $K$. So the module $\mathcal M(Q(K))$ is the classical Alexander module of $K$ (see \cite{BZ} for further reference) and is a knot invariant that will be denoted simply by $\mathcal M(K)$. Analogously, we set $\mathrm{Alex}(Q(K))=\mathrm{Alex}(K)$ and call it the \textit{Alexander quandle} of $K$.

\begin{example}\label{example of a knot}
Let $K$ be a knot in the 3-sphere, and let $Q(K)=[X\colon R]$ be a presentation of the knot quandle as in Theorem \ref{th0}. Then $\mathcal{M}(Q)$ is generated by $\{\xbar{e_{x}}\colon x\in X\setminus \{x_1\}\}$ modulo the relations
\begin{equation}\label{linearized relations statement}
(1-t)\xbar{e_y} +t \xbar{e_x}=\xbar{e_z}
\end{equation}
for each crossing relation $x*y=z$ in $R$, where $x,y,z\in X$. We denote this particular presentation matrix of the Alexander module of a knot by $V_K(t)$. Note that $V_{K}(t)$ depends on the choice of a particular knot diagram and on the labeling of the arcs in this diagram. 
\end{example}

In the case of satellite knots, we recover a classical result by \cite{LM} (see also \cite[Proposition 8.23]{BZ}).

\begin{proposition}\label{representation of the Alex quandle}
Let $S$ be a satellite knot with pattern $P$, companion $C$ and winding number $w$. Then a presentation matrix of $\mathcal{M}(S)$ is
\begin{equation}\label{matrix for pres}
A(t)=\begin{bmatrix}
V_P(t) & 0 \\
0 & V_C(t^w)
\end{bmatrix}.
\end{equation}
\end{proposition}


\begin{proof}
A presentation of the knot quandle of $S$ is provided in Corollary \ref{cor_pres}:
$$Q(S)=\left [S_{P1}, S_{P2}\cup S_{C}\, \colon \, R_{P1},R_{P2}\cup R_{C}\cup \{i_{1}(\mu _{V})=i_{2}(\mu _{U}), i_{1}(\lambda _{V})=i_{2}(\lambda _{U})\}\right ]\;.$$
Let us denote by $x_{i}$ (respectively $y_{i}$) the generators in $S_{P1}$ (respectively $S_{C}$), and let $a_1$ be the only remaining generator of $S_{P2}$, so that the labeling in $$S_{P1}=\{x_1,\ldots, x_d, x_{d+1},\ldots , x_{2d},x_{2d+1},\ldots x_n\}$$ agrees with Figure \ref{fig2}. According to Remark \ref{standard meridians and longitudes}, we can choose $i_{2}(\lambda _{U})=a_1$, $i_{1}(\mu _{V})=y_1$ and $i_{2}(\mu _{U})=\epsilon(x_{i_1})^{\delta_1}\ldots \epsilon(x_{i_n})^{\delta_n}$ with $w=\sum_j \delta_j$. By Proposition \ref{prop1}, the fundamental group of $S$ is isomorphic to the group $G_{P,C}= \pi_1(C)*_{(\pi_1(\partial V), i_1,i_2)}\pi_1(U-P)$.

By Remark \ref{remark on generators}, the elements of $Q(S)$ are of the form $\setof{g(x)}{x\in S_{P1}, \, g\in G_{P,C}}$. Using the notation defined in Proposition \ref{presentation of M(Q)}, we have that the elements of the module $\mathcal{M}(S)$ are of the form $\xbar{e_{g(x)}}$ (recall that we have $\xbar{e_{x_1}}=0$). Using Corollary \ref{homo}, we have that $\xbar{e_{g(x)}}=\pi_\gamma(g)(\xbar{e_x})$.

The relations in the presentation of the group $G_{P,C}$ imply that the image of the action 
$$\rho: G_{P,C}\longrightarrow \aut{(Q(S))}$$
lies in the inner automorphism group of $Q(S)$ and so the image of $\pi_\gamma$ lies in the inner automorphism group of $\textup{Alex}(S)$. 
%
The inner automorphism group of $\textup{Alex}(S)$ is isomorphic to $\dis(\textup{Alex}(S))\rtimes \Z$ and so it is solvable of length $2$ (see \cite[Proposition 2.7]{Principal} and note that its displacement group is abelian). The right multiplication mappings in $\textup{Alex}(S)$ are of the form $((1-t)x,t)$ ( in particular $R_{\xbar{e_{x_1}}}=(0,t)$). 
%

According to \cite[Proposition 3.12]{BZ}, $\lambda_V\in \Gamma_2(\pi_1(C))$ and so $i_{1}(\lambda _{V})=a_1\in \Gamma_2(\pi_1(S))$. Thus $\pi_\gamma(a_1)=1$ and so the relations in $R_{P1}$ of the form $x_{d+i}^{a_1}=x_i$ read as $x_{d+i}=x_i$ for $i=1\ldots d$.  So we can consider just the subset $S_{P1}'$, obtained from $S_{P1}$ by omitting $x_{d+i}$ for $i=1,\ldots,n$.

The relation $y_1=\epsilon(x_{i_1})^{\delta_1}\ldots \epsilon(x_{i_n})^{\delta_n}$ with $w=\sum_j \delta_j$ implies that $\pi_\gamma(y_1)=(\widehat{y}_1,t^w)$, where $\widehat{y}_1=\widetilde{x}$ is an element in the submodule, generated by $\setof{\xbar{e_x}}{S_{P1}'\setminus \{x_1\}}$. The generators in $S_C$ are all conjugate. So $\pi_\gamma(y_i)=(\widehat{y}_i,t^w)$ for some $\widehat{y}_i\in \Gamma_1(\pi_1(S))/\Gamma_2(\pi_1(S))$.  Accordingly, we have that $\mathcal M(S)$ is generated as a module by the set $Z=\setof{\xbar{e_z}}{z\in S_{P_1}'\cup S_C\setminus\{x_1\}}$.

The relations in $R_{C}$ are of the form $y_i y_j y_i^{-1}=y_k$ and so 
$$\pi_\gamma	(y_i y_j y_i^{-1})=((1-t^w)\widehat{y}_i + t^w \widehat{y}_j,t^w)=(\widehat{y}_k,t^w).$$

Hence we have that the generators of the module $\mathcal M(S)$ satisfy the linearized relations obtained from $R_{P1}$ as in \eqref{linearized relation}, that are exactly the same relations collected in $V_P(t)$, the linearized relation $LR_C$, coming from $R_C$, and the relation $\widehat{y_1}=\widetilde{x}$. 

Let $N$ be a $\Lambda$-module, presented as $\langle Z \, \colon \,  V_P(t)\,,\, LR_C,\, \widehat{y_1}=\widetilde{x}\rangle$. Then we have a canonical surjective map
$$N\longrightarrow \mathcal{M}(S)\;,$$
identifying the generators. 

The group $H=N\rtimes \Z$, where the action of $1\in \Z$ is by multiplication by $t$, is generated by $\setof{(z,1),(0,1)}{z\in Z}$ and such generators satisfy the same relations as the generators of the group $\pi_1(S)/\Gamma_2(\pi_1(S))$ (see Proposition \ref{prop1} and note that $a_1\in\Gamma_2(\pi _{1}(S)$ and $x_{d+i}=a_1x_ia_1^{-1}$). Hence there exists a surjective group homomorphism
$$\pi_1(S)/\Gamma_2(\pi_1(S))\longrightarrow H$$
that identifies the generators. Such homomorphism induces a $\Lambda$-module homomorphism between $\mathcal {M} (S)$ and $N$. Therefore $\mathcal {M} (S)$ and $N$ are isomorphic. 

Let us denote by $V$ the module with presentation matrix \eqref{matrix for pres}. We have 
\begin{eqnarray*}
V=\langle Z\setminus\{y_1\}, \, | \,  V_P(t), \, V_C(t^w)\rangle =\langle Z, \, | \,  V_P(t), \, LR_C,\, y_1\rangle.
\end{eqnarray*}
The mapping
$$N \longrightarrow V,\quad z\mapsto z-\widetilde{x}$$
for $z\in Z$ preserves the relations. Therefore it can be extended to a bijective morphism of $\Lambda$-modules (clearly the inverse is defined by $z\mapsto z+\tilde x$). Then $\mathcal{M}(S)$ is presented by the matrix in the statement.
\end{proof}

\subsection{Knot colorings by affine quandles} \label{sub63}

Let $K$ be a knot and $Q$ be a quandle. Denote by $Arc(K)$ the set of arcs in a chosen diagram for $K$. A {\it coloring} of $K$ by $Q$ is a mapping $$c:Arc(K)\longrightarrow Q$$
which respects the crossing relations as in Figure \ref{fig:crossingQ}. A coloring is {\it trivial} if $|Im(c)|=1$. We say that $K$ is {\it colorable} by the quandle $Q$, if $K$ admits a non-trivial coloring by $Q$. Colorings by $Q$ naturally correspond to quandle homomorphisms from the knot quandle $Q(K)$ to $Q$. Since $Q(K)$ is a knot invariant, colorability of a knot does not depend on the choice of a particular knot diagram. Note that $Q(K)$ is a connected quandle and then so is the quandle generated by $Im(c)$. If $Q$ is an affine quandle, then every quandle coloring by $Q$ factors through the map $\mathrm{Alex}(K)\to Q$.

Before introducing the results, we recall the definition of Alexander polynomials of a knot according to \cite{BZ}. 

In a commutative ring $R$, the greatest common divisor (GCD) of a set of elements is defined up to the multiplication by an invertible element, hence in $\Lambda$ it is defined up to the multiplication by a power of $\pm t$. Let $A(t)$ be a $m\times r$ presentation matrix for the Alexander module $\mathcal M(K)$ of a knot $K\in S^3$. We set 
$$\Delta_{n}(K)=\begin{cases}0, \, \text{ if } 0 < m<r-n+1,\\ 
\textup{GCD}\{(r-n+1)-\text{minors of } A(t)\}, \, \text{ if } 0<r-n+1\leq m, \\ 1, \,\text {otherwise,} \end{cases}$$
 and we call $\Delta_{n}(K)$ the $n$-th Alexander polynomial of $K$ (the equality has to be read up to the equivalence mentioned above). The Alexander polynomials are independent of the presentation matrix chosen to compute them - they are invariants of the knot, and $\Delta_1(K)$ is simply called the Alexander polynomial. 


 
The ring $\Lambda$ is a principal ideal domain (PID) and $\mathcal{M}(K)$ is a finitely generated module over $\Lambda$. Hence, using the structure theory of finitely generated modules over a PID (see \cite{BZ}), we can show that the Smith normal form of the presentation matrix of $\mathcal{M}(K)$ is the diagonal matrix with entries $\frac{\Delta_i(K)}{\Delta_{i+1}(K)}$ for $i=1\ldots n-1$, where $n$ is the smallest integer, such that $\Delta_n(K)$ is invertible in $\Lambda$. Accordingly, the 
module structure of $\mathcal{M}(K)$ is the following:
\begin{equation}\label{structure of M}
\mathcal{M}(K)= \frac{\Lambda}{(\Delta_{n-1}(K))}\oplus \frac{\Lambda}{\left(\frac{\Delta_{n-2}(K)}{\Delta_{n-1}(K)}\right)} \oplus \ldots \oplus \frac{\Lambda}{\left(\frac{\Delta_1(K)}{\Delta_2(K)}\right)}.
\end{equation}

In the following, we are using the structure of the Alexander module given in \eqref{structure of M} to show an alternative proof of the main results of \cite{BAE} in the case of knots.

\begin{proposition}\label{coloring by affine}
Let $K$ be a knot. The following are equivalent:

\begin{itemize}
\item[(i)] $K$ is not colorable by any affine quandle.
\item[(ii)] $\Gamma_1(\pi_1(K))$ is perfect.
\item[(iii)] $K$ is not colorable by any solvable quandle.
\item[(iv)] $\Delta_1(K)= 1$.
\end{itemize}

\end{proposition}

\begin{proof}

(i) $\Rightarrow$ (ii) Since $K$ admits a quandle homomorphism $Q(K)\to Alex(K)$, (i) implies that $\mathrm{Alex}(K)$ is trivial. Therefore $\Gamma_1(\pi_1(K))$ is perfect.

(ii) $\Rightarrow$ (iii) If $\Gamma_1(\pi_1(K))$ is perfect, then $\dis(Q(K))$ is also perfect, since the map
$$\pi_1(K)\longrightarrow \inn(Q(K)),\quad x\mapsto R_x$$
can be extended to a surjective group homomorphism, which restricts and co-restricts to $\Gamma_1(\pi_1(K))$ and $\dis(Q(K))$. Assume that $K$ is colorable by a quandle $Q$, i.e. there exists a quandle homomorphism $h:Q(K)	\longrightarrow Q$. We can assume that $h$ is surjective. Then there exists a surjective group homomorphism from $\dis(Q(K))\longrightarrow \dis(Q)$, and so also $\dis(Q)$ is perfect. According to \cite[Lemma 6.1]{CP}, $Q$ is not solvable.
%
%

(iii) $\Rightarrow$ (i) Affine quandles are abelian.

(ii) $\Leftrightarrow$ (iv) It follows directly from the structure of $\mathcal{M}(K)$, displayed in \eqref{structure of M}. Indeed, $\mathcal M (K)$ is trivial if and only if $\frac{\Delta_i(K)}{\Delta_{i+1}(K)}$ is invertible in $\Lambda$ for every $i=1,\ldots,n-1$. Since $$\Delta_1(K)=\Delta_{n-1}(K)\frac{\Delta_{n-2}(K)}{\Delta_{n-1}(K)}\ldots \frac{\Delta_1(K)}{\Delta_2(K)}\;,$$ then this can happen if and only if $\Delta_1(K)$ is invertible, i.e. $\Delta_1(K)=1$.
\end{proof}

 \begin{proposition}\label{coloring2}
 Let $K$ be a knot. If $\Delta_1(K)\neq 1$, then $K$ is colorable by the affine quandle $\aff\left(\frac{\Lambda}{(\Delta_1(K))},t\right)$.
\end{proposition}

\begin{proof}
Denote by $j$ the largest index $i$ such that $\frac{\Delta_i(K)}{\Delta_{i+1}(K)}\neq 1$ (such index exists since $\Delta_1(K)	\neq 1$). The elements $\frac{\Delta_j(K)}{\Delta_{j+1}(K)}$ and $\frac{\Delta_1(K)}{\Delta_{j+1}(K)}$ generate the same ideal in $\Lambda$, since
$$\frac{\Delta_{1}(K)}{\Delta_{j+1}(K)}=\underbrace{\frac{\Delta_{1}(K)}{\Delta_{2}(K)}\frac{\Delta_{2}(K)}{\Delta_{3}(K)}	\ldots \frac{\Delta_{j-1}(K)}{\Delta_{j}(K)}}_{=u}\frac{\Delta_{j}(K)}{\Delta_{j+1}(K)}$$
and $u$ is invertible in $\Lambda$. 
 Accordingly, the mapping
\begin{eqnarray}\label{map}
\mathcal M (K)=\frac{\Lambda}{(\Delta_{n-1}(K))}\oplus \ldots \oplus \frac{\Lambda}{(\frac{\Delta_j(K)}{\Delta_{j+1}(K)})} &\longrightarrow & \frac{\Lambda}{(\Delta_1(K))}\\
 (x_1,\ldots,x_j) &\mapsto & \Delta_{j+1} \cdot x_j+\Delta_1(K)\notag
\end{eqnarray}
is a well-defined module homomorphism with non-trivial image. So, the map \eqref{map} is also a quandle homomorphism from $\mathrm{Alex}(K)$ to $Q=\aff\left(\frac{\Lambda}{(\Delta_1(K))},t\right)$ by Proposition \ref{hom of affine}, and so $K$ admits a non-trivial coloring by a subquandle of $Q$. 
\end{proof}

Note that the image of the map defined in \eqref{map} is the ideal $I$, generated by $\Delta_{j+1}(K)+(\Delta_1(K))$. Then the knot $K$ admits a coloring by the subquandle  $\aff(I,t)$. See \cite{BAE} for some computational results about colorings of knots with crossing number at most $8$.

\section*{Acknowledgements}
 Alessia Cattabriga has been supported by the "National Group for Algebraic and Geometric Structures, and their Applications" (GNSAGA-INdAM) and University of Bologna, funds for selected research topics. Eva Horvat was supported by the Slovenian Research Agency grants P1-0292 and N1-0083.


%
%
%
%
%
%
%
%
%
%
%

\end{document}